\documentclass[11pt, oneside]{amsart} 

\usepackage[english]{babel} 
\usepackage{graphics,color,pgf}
\usepackage{epsfig}
\usepackage[ansinew]{inputenc}
\usepackage[all]{xy}
\usepackage{tikz-cd}
\usepackage{bbm}
\usetikzlibrary{matrix,arrows,decorations.pathmorphing}
\newdir{ >}{!/8pt/@{}*@{>}}
\usepackage{amssymb, amsmath,amsthm, mathtools, amscd}
\usepackage{mathrsfs}
\usepackage[margin=1.5in]{geometry}
\usepackage{tikz}
\usetikzlibrary{automata, arrows.meta, positioning}
\usepackage[pagebackref, pdfborder={0 0 0}
  ,urlcolor=black,hypertexnames=false]{hyperref}

\usepackage[margin=1.5in]{geometry}
  
\makeatletter
\newtheorem*{rep@theorem}{\rep@title}
\newcommand{\newreptheorem}[2]{%
\newenvironment{rep#1}[1]{%
 \def\rep@title{#2 \ref{##1}}%
 \begin{rep@theorem}}%
 {\end{rep@theorem}}}
\makeatother

\newtheorem{intro_thm}{Theorem}

\newtheorem{lemma}{Lemma}[section]
\newtheorem{thm}[lemma]{Theorem} 
\newtheorem{prop}[lemma]{Proposition}

\theoremstyle{definition}
\newtheorem{defn}[lemma]{Definition}
\newtheorem{es}[lemma]{Example}

\theoremstyle{remark}
\newtheorem{oss}[lemma]{Remark}
\newtheorem{nota}[lemma]{Notation}

\newtheoremstyle{TheoremNum}
        {0.2 cm}{0.2 cm}              
        {\itshape}                      
        {}                              
        {}                     
        {.}                             
        { }                             
        {\thmname{\bfseries #1}\thmnote{ \bfseries #3}}
    \theoremstyle{TheoremNum}
   
\newtheorem{rec_thm}{Theorem}

\newtheorem{claim*}[rec_thm]{Claim}

\newcommand\matR{{\mathbb{R}}}

\newcommand{\id}{\mathrm{id}}

\newcommand\calB{{\mathcal B}}

\newcommand\calG{{\mathcal G}}
\newcommand\calH{{\mathcal H}}

\newcommand{\Linf}{\text{L}^{\infty}}

\newcommand{\Ball}{\text{Ball}}

\begin{document}

\title[Boundaries of ergodic groupoids]{Boundaries and equivariant maps for ergodic groupoids}

\author[Filippo Sarti]{Filippo Sarti}
\address{Department of Mathematics, University of Pisa, Italy}
\email{filippo.sarti@dm.unipi.it}

\author[Alessio Savini]{Alessio Savini}
\address{Department of Mathematics, University of Milano Bicocca, Italy}
\email{alessio.savini@unimib.it}

\date{\today.\ \copyright{\ F. Sarti, A. Savini}.}

\begin{abstract}

We give a notion of boundary pair $(\mathcal{B}_-,\mathcal{B}_+)$ for measured groupoids which generalizes the one introduced by Bader and Furman \cite{BF14} for locally compact groups. In the case of a semidirect groupoid $\mathcal{G}=\Gamma \ltimes X$ obtained by a probability measure preserving action $\Gamma \curvearrowright X$ of a locally compact group, we show that a boundary pair is exactly $(B_- \times X, B_+ \times X)$, where $(B_-,B_+)$ is a boundary pair for $\Gamma$. For any measured groupoid $(\mathcal{G},\nu)$, we prove that the Poisson boundaries associated to the Markov operators generated by a probability measure equivalent to $\nu$ provide other examples of our definition. 
  
 Following Bader and Furman \cite{BF:Unpub}, we define algebraic representability for an ergodic groupoid $(\mathcal{G},\nu)$. In this way, given any measurable representation $\rho:\mathcal{G} \rightarrow H$ into the $\kappa$-points of an algebraic $\kappa$-group $\mathbf{H}$, we obtain $\rho$-equivariant maps $\mathcal{B}_\pm \rightarrow H/L_\pm$, where $L_\pm=\mathbf{L}_\pm(\kappa)$ for some $\kappa$-subgroups $\mathbf{L}_\pm<\mathbf{H}$. In the particular case when $\kappa=\mathbb{R}$ and $\rho$ is Zariski dense, we show that $L_\pm$ must be minimal parabolic subgroups. 

  \end{abstract}
  
  \maketitle

  \section{Introduction}

  \subsection*{Motivational background}
  
  The circle, which is the topological boundary of the disk, becomes particularly relevant in the study of harmonic functions in virtue of the well-known \emph{Poisson formula}. Starting from this elementary observation, the notion of boundary has been generalized in many different ways leading to an independent field of research. Roughly speaking, \emph{boundary theory} refers to the study of asymptotic properties of either groups or spaces through the introduction of suitable compactifications with boundaries. Beyond their applications in dynamics, harmonic analysis and geometric group theory, boundaries played a prominent role in \emph{rigidity theory}. Just to mention some milestone results, Mostow Rigidity \cite{mostow68:articolo,Most73}, Margulis-Zimmer Superrigidity \cite{margulis:super,zimmer:annals} and rigidity of maximal representations \cite{iozzi02:articolo,BIW1,bucher2:articolo,Pozzetti,BBIborel} all rely on the existence of certain equivariant maps between boundaries. 
  
  There exist several different notions of boundary in literature. A first example is the \emph{Furstenberg-Poisson boundary} \cite{furstenberg:annals} associated to a random walk on a locally compact group. In virtue of its probabilistic nature, the Poisson boundary is usually hard to realize in a concrete way. Remarkably, when $G$ is a Lie group of non-compact type, the Poisson boundary boils down to the generalized flag space $G/P$, where $P$ is any minimal parabolic subgroup. From the dynamical point of view, Poisson boundaries are both \emph{doubly ergodic} \cite{Kaim} and \emph{amenable} in the sense of Zimmer \cite{zimmer78}. Burger and Monod \cite{burger2:articolo} defined a \emph{strong boundary} as an amenable space which is doubly ergodic when we consider a separable coefficient module as target. Strong boundaries are relevant to compute explicitly bounded cohomology, in particular in low degrees, where no coboundaries appear. Bader and Furman \cite{BF14} generalized both Poisson boundaries and strong boundaries. Their definition of \emph{boundary pair} relies on a weakened version of ergodicity with respect to certain fiberwise isometric actions and it revealed crucial in several contexts, such as the theory of algebraic representability of ergodic actions \cite{bader:duchesne:lecureux:17,bader:furman:compositio} or the study of $\mathrm{CAT}(0)$-spaces with finite telescopic dimension \cite{duchesne:13,duchesne:lecureux:pozzetti:18}.
  
  More recently, the authors \cite{sarti:savini:3,sarti:savini, sarti:savini:2} exploited Bader-Furman boundaries to study rigidity of measurable cocycles. The latter can be viewed as measurable representations of the groupoid associated to a probability measure (class) preserving action. The starting point of our investigation was the work by Burger and Iozzi \cite{burger:articolo}, which allows to realize explicitly numerical invariants coming from bounded cohomology in terms of measurable equivariant maps between boundaries. It is worth noticing that measurable cocycles naturally live in a measurable framework, whereas continuous bounded cohomology heavily relies on the topological properties of the involved groups. Triggered by this discrepancy, in a recent work \cite{sarti:savini:23} the authors set the foundational framework of a new cohomological theory for measured groupoids which provides a more natural description of numerical invariants for cocycles. 
  
  In the attempt to extend the deep connection between boundaries and bounded cohomology beyond locally compact groups, in the present paper we introduce a notion of boundary pair for measured groupoids which mimics the one given by Bader and Furman. We also relate our definition with another notion of boundary for groupoids, namely the Poisson boundary studied by Kaimanovich \cite{Kai05}. We conclude by investigating the existence of equivariant maps for particular groupoid homomorphisms, with the hope that they will reveal relevant in cohomological computations. 
  
  \subsection*{Boundaries for measured groupoids}
  
  A groupoid is a small category where all the morphisms are invertible. More concretely, it is a generalization of the concept of group where not every pair of elements can be composed. A \emph{measured groupoid} is obtained by requiring in addition the existence of a measure which behaves well both under inversion and under left translations (Definition \ref{definition_measured_groupoid}). Classic examples are locally compact groups with the Haar measure, probability measure class preserving actions or measured standard Borel equivalence relations. 
  
  Inspired by Bader-Furman \cite[Definition 2.3]{BF14}, we give a definition of boundary pair for a measured groupoid. Our notion relies on the concepts of \emph{relative isometric ergodicity} (Definition \ref{definition:relative:isometric:ergodicity}) and \emph{amenability} (Definition \ref{definition:amenability}), where we must intend both in the groupoid context.  Relative metric ergodicity, introduced by Bader and Furman \cite{BF14}, is a weakened version of double ergodicity with respect to certain fiberwise isometric actions. Amenability for groupoids is a generalization of Zimmer amenability for measure class preserving actions \cite{zimmer78}. It has several equivalent formulations: for instance, it can be given in terms of the existence of an \emph{equivariant mean} \cite[Definition 3.1.4]{delaroche:renault:libro}, or through the \emph{fixed point property} for fiber affine actions \cite[Definition 4.2.6]{delaroche:renault:libro}. 
  
  A \emph{boundary pair} of a measured groupoid $\mathcal{G}$ is simply a pair of $\mathcal{G}$-spaces $(\mathcal{B}_-,\mathcal{B}_+)$ (Definition \ref{definition_left_space}) such that the semidirect groupoids $\mathcal{B}_-\rtimes \mathcal{G}$ and $\mathcal{B}_+\rtimes \mathcal{G}$ are amenable and the projections $\mathcal{B}_-*\mathcal{B}_+\rightarrow \mathcal{B}_{\pm}$ are relatively isometrically ergodic (Definition \ref{definition_G-boundary}). Here $\mathcal{B}_- \ast \mathcal{B}_+$ is the fiber product with respect to the target maps on the unit space of $\calG$. If it holds $\mathcal{B}_-=\mathcal{B}_+$, we simply refer to a boundary for the groupoid. It is clear that our definition extends the one by Bader and Furman. A more interesting example is given in the case of action groupoids, namely groupoids associated to a probability measure preserving action. 
  
  \begin{intro_thm}\label{theorem_semidirect}
    Let $\Gamma$ be a locally compact second countable group and consider
    a probability measure preserving action $\Gamma\curvearrowright X$ on a
    standard Borel probability space $(X,\mu)$. Let $\mathcal{G}=\Gamma \ltimes X$ be the semidirect groupoid determined by the $\Gamma$-action on $X$. 
    If $((B_-,\beta_-),(B_+,\beta_+))$ is a boundary pair for $\Gamma$,
    then $(B_-\times X,B_+\times X)$ endowed with the product measures
    $(\theta_-,\theta_+)\coloneqq (\beta_-\otimes \mu,\beta_+\otimes \mu)$ is a boundary pair for $\mathcal{G}$. 

\noindent In particular, if $(B,\beta)$ is a $\Gamma$-boundary, then $(B \times X, \beta \otimes \mu)$ is a $\mathcal{G}$-boundary. 
  \end{intro_thm}
  
  The amenability of the semidirect groupoids $\mathcal{B}_\pm\rtimes \mathcal{G}$ follows immediately by the one of the actions $\Gamma\curvearrowright B_{\pm}$. In a similar way, relative metric ergodicity is deduced from the one of the projections $B_-\times B_+\rightarrow B_\pm$, following the line of \cite[Theorem 1]{sarti:savini:3}.
  
  The second non-trivial example is the Poisson boundary of an invariant Markov operator. An \emph{invariant Markov chain} on a measured groupoid $\calG$ is an assignment of a measure to each morphism which is invariant with respect to the natural left action of $\calG$ by pushforward. Given an invariant Markov chain, we can define by integration an \emph{invariant Markov operator} on the space of essentially bounded functions on $\calG$ with respect to some initial distribution. In this way, it is possible to construct a \emph{Markov measure} on the space of \emph{forward trajectories} on $\calG$ (Section \ref{subsec:markov}). The space of ergodic components with respect to the action of the time shift on the space of forward trajectories is called \emph{Poisson boundary} associated to the invariant Markov operator (see Kaimanovich \cite{Kaimanovich1992,Kai05} for more details).
  
  Given a locally compact group $\Gamma$ and a spread out probability measure $\pi$ on $\Gamma$, we denote by $\check{\pi}$ the direct image of $\pi$ through the inversion map. Then the pair of Poisson boundaries $(B,\check{B})$ associated to the random walks generated by $(\pi,\check{\pi})$, respectively, define a boundary pair in the sense of Bader-Furman \cite[Theorem 2.7]{BF14}. If $\mathcal{B}$ is the Poisson boundary of an invariant Markov operator on a measured groupoid $\calG$, Kaimanovich \cite{Kai05} proved that the \emph{Poisson bundle} $\mathcal{B}\rtimes \mathcal{G}$ is an amenable groupoid. If $\pi$ is a probability measure equivalent to the integrated Haar system on $\mathcal{G}$, we still denote by $\check{\pi}$ the inverse of $\pi$, as in the case of groups. In this paper we show that the pair $(\check{\mathcal{B}},\mathcal{B})$  of Poisson boundaries associated to the Markov operators generated by $(\check{\pi},\pi)$, respectively, is a boundary pair in our sense (Section \ref{subsec:gboundary}). More precisely, we show that the projections $\check{\mathcal{B}} \ast \mathcal{B} \rightarrow \check{\mathcal{B}}$ and $\check{\mathcal{B}} \ast \mathcal{B} \rightarrow \mathcal{B}$ are relatively isometrically ergodic. 
   
  \begin{intro_thm}\label{theorem_poisson_boundary}
    Let $(\calG,\nu)$ be a measured groupoid with unit space $(X,\mu)$ and let $\pi$ be a probability measure equivalent to $\nu$ such that $t_*\pi=s_*\pi=\mu$, where $t$ and $s$ are the target and the source on $\mathcal{G}$, respectively. Denote by $\check{\pi}$ the direct image of $\pi$ under the inverse map. The pair $((\mathcal{B},\theta_{\mu}),(\check{\mathcal{B}},\check{\theta}_{\mu}))$ of Poisson boundaries associated to the Markov operators generated by $\pi$ and $\check{\pi}$, respectively, with initial distribution $\mu$ is a boundary pair for $\mathcal{G}$.

\noindent In particular, if $\pi$ is symmetric, namely $\pi=\check{\pi}$, then $(\mathcal{B},\theta_{\mu})$ is a $\mathcal{G}$-boundary.
  \end{intro_thm}
  
  Our proof heavily relies on the strategy exploited by Bader and Furman \cite{BF14}. In a preliminary result (Proposition \ref{prop:stationary}) we show a \emph{stationarity} property which generalizes the one already known for groups. This will be crucial to construct by hand the lift required by relative isometric ergodicity. 
  
  \subsection*{Algebraic representability of ergodic groupoids and equivariant maps}
  It is well known that the Zariski closure of the image of a group homomorphism into an algebraic group is still a subgroup of the target. It should be clear that for a groupoid homomorphism, this cannot be true any longer. In the context of measurable cocycles with algebraic targets, Zimmer \cite{zimmer:libro} defined the notion of \emph{algebraic hull} as the smallest algebraic group containing the image of any representative in a fixed cohomology class. Following the line traced by Zimmer, Bader and Furman \cite{BF:Unpub} introduced the concept of \emph{algebraic representability} of an ergodic action. Given a measurable cocycle with an algebraic target, they were able to translate a generic ergodic action on a standard Borel space into some algebraic action on certain quasi-projective quotients. The latter is a crucial step both in the study of equivariant maps \cite{BF14,bader:duchesne:lecureux:17} and in some rigidity statement relying on the existence of suitable invariant projective measures \cite{BFMS21,BFMS:Unpub,Sav:proj}. 
  
  Inspired by the relevance of algebraic representability, we generalize it for any ergodic groupoid. In this context, the action of a locally compact group on a standard Borel space is substituted by the action of the groupoid on its unit space (Theorem \ref{teor:initial:object}). In this way, we are able to prove an existence result for measurable equivariant maps from the boundaries of a measured groupoid. 
  
  \begin{intro_thm}\label{thm_equivariant_map}
    Let $\rho:\mathcal{G} \rightarrow H$ be any measurable representation of an ergodic groupoid into the $\kappa$-points of an algebraic $\kappa$-group $\mathbf{H}$. Let $((\mathcal{B}_-,\theta_-),(\mathcal{B}_+,\theta_+))$ be a boundary pair for $\calG$. Then there exist $\kappa$-subgroups $\mathbf{L}_\pm<\mathbf{H}$ and measurable $\rho$-equivariant maps $\mathcal{B}_{\pm} \rightarrow H/L_\pm$, where $L_\pm=\mathbf{L}_\pm(\kappa)$. 
    \end{intro_thm}
  
  The main point in the proof of the previous theorem is to show that $\mathcal{B}_\pm \rtimes \calG$ are ergodic groupoids, but this is a direct consequence of relative isometric ergodicity (Proposition \ref{proposition_properties}). 

In the particular case $\kappa=\mathbb{R}$, it is natural to ask under which conditions the equivariant maps given by Theorem \ref{thm_equivariant_map} are actually \emph{Furstenberg maps}, namely when the subgroups $L_\pm$ are \emph{minimal parabolic} subgroups. For both group representations \cite{BF14} and measurable cocycles \cite{sarti:savini:3} it is sufficient to require \emph{Zariski density}. We will extend the previous results to the setting of ergodic groupoids. 
 
  \begin{intro_thm}\label{theorem_furstenberg_map}
Let $\rho: \mathcal{G}\rightarrow H$ be a measurable Zariski dense representation of an ergodic groupoid $(\calG,\nu)$ into the real points $\text{H}=\textbf{\textup{H}}(\mathbb{R})$ of a real algebraic group. Let $((\mathcal{B}_-,\theta_-),(\mathcal{B}_+,\theta_+))$ be a boundary pair for $\mathcal{G}$. Then there exist $\mathcal{G}$-equivariant measurable maps $\mathcal{B}_\pm \rightarrow H/P_\pm$, where $P_\pm<H$ are minimal parabolic subgroups.

\noindent In particular if $(\mathcal{B},\theta)$ is a $\mathcal{G}$-boundary, then there exists an equivariant measurable map $\mathcal{B} \rightarrow H/P$.
  \end{intro_thm}
  
Given equivariant maps $\phi:\mathcal{B}_\pm \rightarrow H/L_\pm$, for any unit $x \in X$, the $x$-\emph{slice} of $\phi_\pm$ is the restriction of $\phi_\pm$ to the fiber of the target on $\mathcal{B}_\pm$. Under the assumption of Zariski density, we start showing that the slices of an equivariant map are essentially Zariski dense (Lemma \ref{lemma:Zarski:density:slices}). This fact, combined with both relative metric ergodicity and algebraic representability, allows to conclude that $L_\pm$ are indeed minimal parabolic subgroups.

  \subsection*{Plan of the paper} In Section \ref{sec:groupoids} we introduce the basic aspects about groupoids. We recall their algebraic definition, the notion of morphism and similarity, then we move to the measured setting. We conclude by characterizing ergodic groupoids. In Section \ref{sec:strong:boundaries} we give the definition of boundary pair for a measured groupoid. The first non-trivial example is given for a locally compact group $\Gamma$ acting in measure preserving way on a probability space $X$: in that case the boundary pair of the groupoid is the boundary pair of $\Gamma$ multiplied by $X$ (Theorem \ref{theorem_semidirect}). Another relevant example of boundary pair is given by the Poisson boundaries associated to the Markov operators generated by a quasi-invariant probability measure on a groupoid. We start Section \ref{sec:poisson} by recalling the construction of the Poisson boundary associated to a Markov operator. Then we prove that the measure on the boundary is stationary with respect to the groupoid action. Finally we conclude the section by showing Theorem \ref{theorem_poisson_boundary}. Section \ref{section_algebraic_representability} is devoted to algebraic representability of ergodic groupoids: after some some technical results, we introduce the notion of algebraic representability. The latter is crucial to prove the existence of equivariant maps for boundary pairs (Theorem \ref{thm_equivariant_map}) . Under the assumption of Zariski density, those maps land in the Furstenberg boundary of the target group (Theorem \ref{theorem_furstenberg_map}). 
  
  \subsection*{Acknowledgements} We would like to thank Bruno Duchesne, Camille Horbez and Jean L\'ecureux for useful conversations during the visiting of the first author to the Laboratoire de Math\'ematiques d'Orsay, funded by the ERC StG project \emph{``Artin groups, mapping class groups and $Out(F_n)$: from geometry to operator algebras via measure equivalence"} (Artin-Out-ME-OA, Grant 101040507).\\
  We are also grateful to Uri Bader, Vadim Kaimanovich and Jean L\'ecureux for being interested in the preliminary version of this paper and for their suggestions which increased the quality of the manuscript.  The first author is funded by the European Union - NextGenerationEU under the National Recovery and Resilience Plan (PNRR) - Mission 4 Education and research - Component 2 From research to business - Investment 1.1 Notice Prin 2022 - DD N. 104 del 2/2/2022, from title \emph{``Geometry and topology of manifolds"}, proposal code 2022NMPLT8 - CUP J53D23003820001
  and partially supported by INdAM--GNSAGA Project CUP E55F22000270001.

\section{Preliminaries}\label{sec:groupoids}

In this first part we recall some necessary material that we will need throughout the paper. We start with a brief overview about groupoids, with particular attention on measured ones. 
Before going into the measured framework, we first recall some basic notions about groupoids as purely algebraic structures. 
We essentially adopt the same notation as in \cite{sarti:savini:23}, but for more details we refer either to Muhly \cite{muhly} or to Anantharaman-Delaroche and Renault \cite{delaroche:renault}.

\begin{defn}\label{def_groupoid}
  A \emph{groupoid} is a small category whose morphisms are invertible.
  \end{defn}

A groupoid $\mathcal{G}$ is determined by both its set of objects (or its \emph{unit space}), denoted by $\mathcal{G}^{(0)}$,
and by its set of morphisms that we denote again by $\mathcal{G}$ with a slight abuse of notation.
The natural \emph{source} and \emph{target} maps
are denoted by $s:\mathcal{G}\rightarrow \mathcal{G}^{(0)}, s(g):=g^{-1}g$ and $t:\mathcal{G}\rightarrow \mathcal{G}^{(0)}, t(g)=gg^{-1}$, respectively. 
The set of \emph{composable pairs} is 
$$\calG^{[2]}\coloneqq \{ (g_1,g_2)\in \mathcal{G}\,|\, t(g_2)=s(g_1)\}\subset \mathcal{G}\times \mathcal{G}\,$$
and the \emph{composition map} is
$$\calG^{[2]}\rightarrow \mathcal{G}\,,\;\;\; (g_1,g_2)\mapsto g_1g_2,$$
where we understand that 
$s(g_1g_2)=s(g_2)$ and $t(g_1g_2)=t(g_1)$.
Any element $g\in \mathcal{G}$ admits an \emph{inverse} $g^{-1}$ that satisfies the following cancellation rules
$$g^{-1}gh=h, \,\;\;\; kgg^{-1}=k,$$
for every $h,k\in \mathcal{G}$ with $t(h)=s(g)$ and $t(g)=s(k)$.
For any $x\in \mathcal{G}^{(0)}$, the fiber with respect to 
the source (respectively to the target) is denoted by $\mathcal{G}_x$ (respectively by $\mathcal{G}^x$).

\begin{es}\label{example_groupoids}
   We list some basic examples of groupoids.
\begin{itemize}
  \item[(i)] Any group $G$ is a groupoid, with unit space $G^{(0)}=\{1_G\}$. Both the source and the target maps are trivial.
  \item[(ii)] An equivalence relation $\mathcal{R}\subset X\times X$ on a set $X$ has an associated groupoid, denoted by $\mathcal{G}_{\mathcal{R}}$. Precisely, the unit space is $\mathcal{G}_{\mathcal{R}}^{(0)}=X$ and the set of morphisms is $\mathcal{G}_{\mathcal{R}}=\mathcal{R}$. Here the source and the target maps are 
  $s(y,x)=x$ and $ t(y,x)=y$, respectively, and the composition of two elements 
  $(z,y), (y,x)$ is $(z,x)$. 
  Finally, the inverse of the element $(y,x)$ is $(x,y)$.
  \item[(iii)] A (left) group action $G\curvearrowright X$ gives rise to the \emph{action groupoid} (or \emph{semidirect groupoid}) $G\ltimes X$, whose set of units is $
  (G\ltimes X)^{(0)}= X$ and whose set of composable pairs is
  $$(G\ltimes X)^{[2]}=\{((h,y),(g,x))\in (G\times X)^2\, | \,  gx=y\}.$$
  The composition is given by the formula 
  \begin{equation*}
    (h,gx)(g,x)=(hg,x),
  \end{equation*}
  and the inverse is defined as
  \begin{equation*}
    (g,x)^{-1}=(g^{-1},gx).
  \end{equation*} 
  The target map is $t(g,x)=gx$ and the source map is $s(g,x)=x$.
  We notice that right-action groupoids can be similarly defined.
  \item[(iv)] Let $\calG$ be a groupoid and let $S$ be a set with a surjection $t_S:S\rightarrow \calG^{(0)}$.  If we define 
 $$\calG * S\coloneqq \{ (g,s)\in \calG\times S\, | \, s_{\calG}(g)=t_{S}(s)  \},$$
where $s_{\calG}$ is the source on $\calG$, we say that $\calG$ \emph{acts  on the left} on $S$ (or that $S$ is a \emph{left} $\mathcal{G}$-\emph{space})
if there exists a map $$\calG*S\rightarrow S,\;\;\;\; (g,s)\mapsto gs $$
 satisfying the following conditions:
 \begin{itemize}
 \item $(gh)s=g(hs)$, whenever $(gh,s)\in \calG* S$ and $(g,hs) \in \calG* S $;
 \item $t_S(gs)=t_{\calG}(g)$, with $t_{\calG}$ target on $\calG$, whenever $(g,s) \in \calG* S$;
 \item $gg^{-1}s =g^{-1} gs=s$, whenever $(g,s)\in \calG*S$.
 \end{itemize}
Such a left action gives rise to a groupoid denoted by 
$$
S \rtimes \mathcal{G}:= S \ast \mathcal{G}=\{ (s,g) \in S \times \mathcal{G} \ | \ t_{S}(s)=t_{\mathcal{G}}(g) \}
$$
and called again \emph{semidirect groupoid}. The set of composable pairs is given by 
$$
(S\rtimes \mathcal{G})^{[2]}:=\{ ((s_1,g_1),(s_2,g_2)) \in (S \rtimes \mathcal{G})^2 \ | \ s_2=g_1^{-1}s_1\}
$$
and the composition map is $(s_1,g_1)(g^{-1}_1 s_1,g_2)=(s_1,g_1g_2)$. The inverse of the element $(s,g)$ is given by $(g^{-1}s,g^{-1})$. In this way one can see that the source and the target of the element $(s,g)$ are $g^{-1}s$ and $s$, respectively. 
\end{itemize}
\end{es}

\begin{oss}
It should be clear that also $\mathcal{G} \ast S$ admits a natural structure of groupoid, denoted by $\mathcal{G} \ltimes S$, extending the one given by Example \ref{example_groupoids} (iii). Nevertheless, when we are going to introduce the notion of Haar system with respect to the target map, we will see that, given a Haar system on $\mathcal{G}$, it is easier to define a Haar system for $S \rtimes \mathcal{G}$ rather than for $\mathcal{G} \ltimes S$. 
\end{oss}

Maps between groupoids can be defined, from the categorical perspective, as functors between two small categories. 
However, for our purposes, we prefer to give an equivalent but more practical definition, which justifies why those functions generalize groups homomorphisms.
\begin{defn}\label{definition_homomorphism_homotopy_similarity}
A \emph{homomorphism} between two groupoids $\calG$ and $\calH$ is a set-theoretic
map $f:\calG\rightarrow \calH$ such that if $(g,h)\in \calG^{[2]} $ then $(f(g),f(h))\in \calH^{[2]}$ and it holds that
$$f(gh)=f(g)f(h)\,.$$
An invertible groupoid homomorphism is a groupoid \emph{isomorphism}. Whenever $\calG$ and $\calH$ are related by an isomorphism, we say that they are \emph{isomorphic} and we write $\calG\cong \calH$.

Given two homomorphisms $f_0,f_1:\calG\rightarrow \calH$, a \emph{similarity} between $f_0$ to $f_1$ is a map $h:\calG^{(0)}\rightarrow \calH$ such that for every unit $x\in \calG^{(0)}$ one has $s(h(x))=f_0(x)$, $t(h(x))=f_1(x)$ and for every $g\in \calG$ it holds
\begin{equation}\label{equation_homotopy}
h(t(g)) f_0(g)h(s(g))^{-1}= f_1(g)\,.
\end{equation}
If such a similarity exists, we write $f_0\simeq f_1$. The composability of the left-hand side of Equation \eqref{equation_homotopy} descends directly from the definition. 

Two groupoids $\calG,\calH$ are \emph{similar} if there exist homomorphisms 
$f_0:\calG\rightarrow \calH$ and $f_1:\calH\rightarrow \calG$ such that 
$$f_1\circ f_0\simeq \id_{\calG}\,,\;\;\; f_0\circ f_1\simeq \id_{\calH}\,.$$
Here the notation $\mathrm{id}$ refers to the identity map. 
\end{defn}

\begin{es}
  Basic examples of groupoids homomorphisms are groups homomorphisms or, more generally, 1-cocycles.
  Given a group action $\Gamma\curvearrowright X$ and a group $\Lambda$, a \emph{1-cocycle} is a groupoid homomorphism from $\Gamma \ltimes X$ to $\Lambda$, namely a map $f:\Gamma \times X\rightarrow \Lambda$ such that
  $$f(\gamma_1\gamma_2,x)=f(\gamma_1,\gamma_2x)f(\gamma_2,x)$$ for every $\gamma_1,\gamma_2\in \Gamma$ and $x\in X$. 
  Such kind of objects arise naturally in the measurable context, namely as examples of measurable homomorphisms between measured groupoids (see Definition \ref{definition_measured_groupoid}). In this environment they are called \emph{Borel 1-cocycles} or \emph{measurable cocycles}.
\end{es}

Now we want to introduce a suitable structure of measure spaces on groupoids. The first step is the following
\begin{defn}
A \emph{measurable groupoid} (or \emph{Borel groupoid}) is a groupoid endowed with a $\sigma$-algebra such that the composition and the inverse are Borel maps (here $\calG^{[2]}$ is endowed with the Borel structure inherited by $\calG\times \calG$).

\end{defn}

All the notions introduced in Definition \ref{definition_homomorphism_homotopy_similarity} adapts to the the measurable setting by adding the further request that the involved maps are measurable.
From now on, since we will only work in the measurable context, we will tacitly assume that all maps are measurable. Furthermore, unless otherwise mentioned, all measure spaces are assumed to be \emph{standard Borel}.

Once a $\sigma$-algebra on a groupoid is fixed, our next aim is to introduce a good notion of \emph{measure}. 
The source of inspiration comes from the framework of locally compact groups, where we have the Haar measure, which is invariant with respect to the group action. 
In the case of groupoids, the basic fact that any $g\in \mathcal{G}$ defines a 
bijection 
$$\mathcal{G}^{s(g)}\rightarrow \mathcal{G}^{t(g)},\;\;\; h\mapsto gh$$
suggests that, in order to mimic the behaviour of the Haar measure, one should require the existence of a family of measures supported on the $t$-fibers satisfying a suitable invariance property.
\begin{defn}\label{definition_borel_system}
  A \emph{Borel Haar system of measures} for the target map $t:\calG\rightarrow \calG^{(0)}$ of a measurable groupoid $\mathcal{G}$ is a family $\rho=\{\rho^x\}_{x\in \calG^{(0)}}$ of $\sigma$-finite measures on $\mathcal{G}$, where $\rho^x(\mathcal{G}\setminus \calG^x)=0$ for every $x\in X$, and such that
  \begin{itemize}
    \item for every non-negative measurable map $f$ on $\calG$, the function 
    $$x\mapsto \rho^x(f):=\int_{\mathcal{G}} f(g)d\rho^x(g)$$ is measurable;
    \item $\rho$ is \emph{left} $\calG$-\emph{invariant}, namely
 $$ \int_{\calG} f(g h)d \rho^{s(g)} (h)=\int_{\mathcal{G}} f(h) d \rho^{t(g)}(h).$$
The above equation can be rewritten as follows
\begin{equation}\label{equation invariance system target}
    g\rho^{s(g)}=\rho^{t(g)}.
    \end{equation}
  \end{itemize}
  \end{defn}

\begin{oss}\label{rem_disintegrazione_source}
We can define a Borel Haar system of measures with respect to the \emph{source} map in a similar way. In that case, we should require that the system $\{\rho_x\}_{x \in X}$ is \emph{right} $\mathcal{G}$-\emph{invariant}, namely it holds 
$$
\rho_{t(g)}g=\rho_{s(g)}.
$$
In order to pass from a left invariant system to a right invariant one, it is sufficient to exploit inversion. For more details we refer to \cite{delaroche:renault}. 
\end{oss}

\begin{oss}\label{remark_Borel_system}
  The notion of Borel system can be given in the more general setting of a Borel projection $\pi:Y\rightarrow X$ between standard Borel spaces. Precisely, it consists of a family of measures $\{\rho^x\}_{x\in X}$ with $\rho^x(Y\setminus \pi^{-1}(x))=0$ and such that the function
  $$x\mapsto \rho^x(f)$$ is measurable for every non-negative measurable map $f$ on $Y$. 
  Borel systems of projections $Z\rightarrow Y$ and $Y\rightarrow X$ can be composed in order to get a Borel system for the composition $Z\rightarrow Y\rightarrow X$. Moreover,
  if $t_Y:Y\rightarrow X$ and $t_Z:Z\rightarrow X$ are Borel projections between standard Borel spaces endowed with Borel systems $\{\tau^x\}_{x\in X}$ and $\{\theta^x\}_{x\in X}$, respectively, then we denote by 
  $$Y*Z\coloneqq \{(y,z)\in Y\times Z\,|\, t_Y(y)=t_Z(z)\}$$ 
  their fiber product. The collection of product measures $\{\tau^x\otimes\theta^x\}_{x\in X}$ is a Borel system for the obvious projection $Y*Z\rightarrow X$.
\end{oss}

Beyond the invariance with respect to the $\mathcal{G}$-action, the family of measures defining a Borel Haar system must behave well also with respect to the inversion.

  \begin{defn}\label{definition_invariant_borel_system}
    Given a measurable groupoid $\calG$ endowed with a Borel Haar system of measures $\rho=\{\rho^x\}_{x \in \calG^{(0)}}$ for the target map $t$, a measure $\mu$ on $\calG^{(0)}$ is \emph{quasi-invariant with respect to $\rho$} if the composition $\mu \circ \rho$ defined by
    \begin{equation} \label{equation_convolution_measures}
    (\mu \circ \rho)(f)=\int_{\calG^{(0)}}\left( \int_{\calG}  f(g)d\rho^x(g)\right)d\mu(x)
    \end{equation} 
     is equivalent to its direct image under the inverse map $g\mapsto g^{-1}$. In this case we say that $\mu \circ \rho$ is \emph{quasi-symmetric} and it is \emph{symmetric} if it is invariant. 
    \end{defn}
    
    The composition $\mu \circ \rho$ naturally defines a measure on $\calG$ whose class $C$ is \emph{symmetric}, namely such that any representative $\nu \in C$ is equivalent to the direct image $\nu^{-1}$ under the inverse map $g\rightarrow g^{-1}$. 
Moreover, $C$ is also \emph{left invariant}, which means that it contains a probability measure whose \emph{$t$-disintegration} is left quasi-invariant. 
A $t$-disintegration of a probability measure $\nu$ is a measurable map from the unit space $\calG^{(0)}$ to the space of probability measures on $\calG$, sending $x \mapsto \nu^x$, such that $\nu^x(\calG^x)=1$ and 
\begin{equation}\label{equation_isom_disintegration}
  \int_{\calG} f(g)d\nu(g)=\int_{\calG^{(0)}} \left(\int_{\calG} f(g)d\nu^x(g)\right) d\mu(x) \ ,
  \end{equation}
  for any $f \in \mathrm{L}^1(\calG)$.
   In our context, by Hahn disintegration theorem (see for instance Effros \cite[Lemma 4.4]{Eff66} or Hahn \cite[Theorem 2.1]{Hahn}), any probability measure $\nu$ with $t_*\nu=\mu$ can be disintegrated with respect to the target $t:\calG \rightarrow \calG^{(0)}$.
    We say that the $t$-disintegration is \emph{left quasi-invariant} if $g\nu^{s(g)}$ and $\nu^{t(g)}$ define the same measure class, for $\nu$-almost every $g \in \calG$.

    \begin{defn}\label{definition_measured_groupoid}
    A \emph{measured groupoid} is a Borel groupoid $\calG$ endowed with a symmetric and left invariant measure class $C$.
    \end{defn}

    \begin{es}\label{example_groupoid_action_measurable}
      Let $G$ be a locally compact group with left Haar measure $\rho$ and $(X,\mu)$ be a Lebesgue $G$-space, that is a standard Borel probability space on which $G$ acts by preserving the measure class. If we consider the Borel Haar system given by 
$$
\rho^x:=\rho \otimes \delta_x, \ \ \ x \in X,
$$
we can define the composition
$$
\mu \circ \rho=\int_X \rho^x d\mu(x),
$$
which boils down to the product measure $\rho \otimes \mu$. The measure class of $\rho \otimes \mu$ is quasi-symmetric and left invariant, turning the semidirect groupoid $G \ltimes X$ into a measured groupoid.       
      \end{es}

\begin{es}
Suppose that the action map of Example \ref{example_groupoids} (iv) is measurable. As a consequence we have that the semidirect groupoid $S \rtimes \mathcal{G}$ is Borel. 

Furthermore, if $(\rho^x)_{x\in \mathcal{G}^{(0)}}$ is a Borel Haar system for $t_{\mathcal{G}}:\mathcal{G}\rightarrow \mathcal{G}^{(0)}$,  we can lift it to a Borel Haar system $\{\rho^s\}_{s\in S}$ for $t:S\rtimes \mathcal{G}\rightarrow S$ by setting
$$
\rho^s:=\rho^{t_S(s)} \otimes \delta_s, \ \ \ s \in S. 
$$
Given a measure $\tau$ on $S$, we say that 
$\tau$ is \emph{quasi-invariant} with respect to $\{\rho^s \}_{s \in S}$ if the composition $\tau \circ \rho$ defined by 
\begin{equation}\label{equation_measure}
 (\tau \circ \rho)(f)\coloneqq\int_S\left(\int_{S\rtimes \mathcal{G}} f(s,g) d\rho^{t_S(s)}(g) \right)d\tau(s)
\end{equation}
is equivalent to its image under the map 
$(s,g)\mapsto(g^{-1}s,g^{-1})$. Equivalently, $\tau$ is quasi-invariant if $S\rtimes \mathcal{G}$ is a measured groupoid when endowed with the measure class defined by $\tau \circ \rho$ \cite[Definition 3.1.1]{delaroche:renault}. When $(S,\tau)=(\calG^{(0)},\mu)$, this gives us precisely the quasi-invariance of $\mu$.
\end{es}

\begin{defn}\label{definition_left_space}
Let $(\mathcal{G},\nu)$ be a measured groupoid. 
A \emph{Lebesgue} $\mathcal{G}$-\emph{space} is a (left) $\mathcal{G}$-space $S$ with a quasi-invariant measure $\tau$. 
\end{defn}

\begin{es}\label{example_action1}
Let $(\mathcal{G},\nu)$ be a measured groupoid. Then 
$(\mathcal{G},\nu)$ is naturally a Lebesgue $\mathcal{G}$-space, since 
\begin{equation}\label{equation_right_action_groupoid}
  \mathcal{G}^{(2)}\coloneqq \mathcal{G}\rtimes\mathcal{G}=\{(g,h)\in \mathcal{G}^2\,|\, t(g)=t(h)\}
\end{equation}
has a natural structure of measured groupoid inherited by the one of $\mathcal{G}$. More generally, the fiber product
$$\mathcal{G}^{(n)}\coloneqq\mathcal{G}*\cdots *\mathcal{G}=\{(g_1,g_2,\cdots g_n)\in 
\mathcal{G}^n\,|\, t(g_i)  =t(g_j)\}$$
endowed with the measure $\nu^{(n)}=\nu\widehat{\otimes}\cdots \widehat{\otimes} \nu$ (see \cite[Section 4.2]{sarti:savini:23}) is a Lebesgue $\mathcal{G}$-space. 
\end{es}

\begin{oss}
Given a measured groupoid structure on $S \rtimes \mathcal{G}$, a measure class turning the groupoid $\mathcal{G}\ltimes S$ into a measured one is easily obtained by pushing forward the composition $\tau \circ \rho$ of Equation \eqref{equation_measure} with respect to the map
  $$S\rtimes \mathcal{G}\rightarrow\mathcal{G}\ltimes S\,,\;\;\; (s,g)\mapsto (g,g^{-1}s).$$
\end{oss}

\begin{oss}
  Consider two Borel projections $Y\rightarrow X$ and $Z\rightarrow X$ as in Remark \ref{remark_Borel_system}.
  If $\tau$ and $\theta$ are measures on $Y$ and $Z$, respectively, that disintegrate with respect to $t_Y$ and $t_Z$ into Borel systems $\{\tau^x\}_{x\in X}$ and $\{\theta^x\}_{x\in X}$, respectively, then the formula 
  $$\tau\widehat{\otimes} \theta \coloneqq \int_X \tau^x\otimes\theta^x d\mu(x)$$
  defines a measure on $Y*Z$.
  In particular, $X=\mathcal{G}^{(0)}$ is the unit space of a measured groupoid $(\mathcal{G},\nu)$ and $(Y,\tau)$, $(Z,\theta)$ are Lebesgue $\mathcal{G}$-space in the sense of Definition \ref{definition_left_space}, then $(Y*Z,\tau\widehat{\otimes} \theta)$ is also a Lebesgue $\mathcal{G}$-space.
\end{oss}

A peculiar feature of boundaries from the dynamical viewpoint is their amenability. For instance, Poisson boundaries of locally compact groups are amenable spaces with respect to the group action. This property can be expressed using the notion of \emph{amenable action} introduced by Zimmer \cite{zimmer78}.  The next definition is precisely a generalization of this concept for a generic measured groupoid. As it happens for both groups and actions, also in the case of groupoids amenability admits several equivalent formulations. Since we are not going to use it directly, we only recall one possible definition and we refer to the book by Anantharaman-Delaroche and Renault \cite[Chapter 3]{delaroche:renault} for more details.  

\begin{defn}\label{definition:amenability}
A measured groupoid $(\mathcal{G},\nu)$ is \emph{amenable} if there exists a $\mathcal{G}$-invariant positive linear functional
$m:\Linf(\mathcal{G},\nu)\rightarrow \Linf(\mathcal{G}^{(0)},\mu)$ of norm one. 
\end{defn}

We conclude this section by introducing the notion of ergodicity for a measured groupoid. 

\begin{defn}\label{definition_ergodic_groupoid}
  Given a subset $U\subset \mathcal{G}^{(0)}$ of a measured groupoid $(\mathcal{G},\nu)$, its \emph{saturation} is 
  $$\mathcal{G}U\coloneqq \{t(g)\,|\, s(g)\in U\}\subset \mathcal{G}^{(0)}\,.$$
  We say that a Borel set $U\subset \mathcal{G}^{(0)}$ is \emph{invariant} if $\mathcal{G}U=U$.  A measured groupoid $(\mathcal{G},\nu)$ is \emph{ergodic} if any invariant Borel set $U\subset \mathcal{G}^{(0)}$ is either null or co-null.
  \end{defn}

In the sequel we will actually exploit a characterization of ergodicity in terms of measurable invariant functions on the unit space. Given a measured groupoid $(\mathcal{G},\nu)$ and a countably separated Borel space $Y$, a measurable function $f:X=\mathcal{G}^{(0)} \rightarrow Y$ is \emph{quasi invariant} if it satisfies
$$
f(t(g))=f(s(g)),
$$
for almost every $g \in \mathcal{G}$. As noticed by Ramsay \cite{ramsay}, the groupoid $(\mathcal{G},\nu)$ is ergodic if and only if any measurable quasi invariant functions as above are essentially constant. For instance, given a Lebesgue $\mathcal{G}$-space $(S,\tau)$, the semidirect groupoid $S \rtimes \mathcal{G}$ is ergodic if and only if any quasi invariant function $f:S \rightarrow Y$ into a countably separated space is essentially constant. 

The characterization of ergodicity in terms of quasi invariant functions will be crucial to introduce the notion of algebraic representability of ergodic groupoids in Section \ref{section_algebraic_representability}. 

\section{Boundaries for groupoids}\label{sec:strong:boundaries}

In this section we present our notion of boundary pair for a measured groupoid. We are going to mimic the definition given by Bader and Furman \cite{BF14} in the case of locally compact groups. For a semidirect groupoid associated to a probability measure-preserving action, we show that we can obtain a boundary pair multiplying a boundary pair of the acting group with the space of units. 

\begin{nota}
  We fix the following setting:
  \begin{itemize}
  \item $(\calG,\nu)$ is a measured groupoid, 
  where $\nu$ is a quasi-symmetric and left quasi-invariant probability measure. 
  The unit space $X=\calG^{(0)}$ is endowed with a probability measure 
  $\mu$. We denote by $t:\calG\rightarrow X$ the target map and by 
  $(\nu^x)_{x\in X}$ the disintegration of $\nu$ with respect to $t$.
  \item For any Lebesgue $\mathcal{G}$-space $(Y,\tau)$ with $(t_Y)_*\tau=\mu$ and
 any $x\in \mathcal{G}^{(0)}$ we denote by $Y^x$ the fiber over $x$ and by $\{\tau^x\}$ the $t_Y$-decomposition of $\tau$
  \end{itemize}
  \end{nota}

We start with the following:
\begin{defn}
Consider a Lebesgue $\mathcal{G}$-space $(Y,\tau)$ with target $t_Y:Y \rightarrow X$ and a Borel $\mathcal{G}$-space $Z$ with target $t_Z:Z \rightarrow X$.  A \emph{$\mathcal{G}$-equivariant} Borel map $\varphi:Y \rightarrow Z$ is a target preserving Borel map, namely $t_Z \circ \varphi=t_Y$, which satisfies
$\varphi(g y)=g\varphi(y)$ for every $g \in \mathcal{G}$ and $\tau^x$-almost every $y \in Y^{s(g)}$. Two $\mathcal{G}$-equivariant Borel functions will be considered equivalent if they coincide $\tau$-almost everywhere. A \emph{$\mathcal{G}$-map} will be such an equivalence class of $\mathcal{G}$-equivariant Borel maps (but we will always consider a representative of the equivalence class). 

A $\mathcal{G}$-map between Lebesgue $\mathcal{G}$-spaces will be assumed to be measure class preserving.  
\end{defn}

\begin{defn}\label{definition_fiberwise_isometric_action}
Let $Y$ and $Z$ be standard Borel spaces and $q:Y\rightarrow Z$ be a 
Borel map. Then a \emph{metric} on $q$ is the datum of a Borel function
$$d:Y *_q Y\rightarrow \matR_{\geq 0},$$ where 
$Y *_q Y$ is the fiber
 product with respect to $q$ and 
such that 
for every $z\in Z$ the function 
$$d_z\coloneqq d_{q^{-1}(z)\times q^{-1}(z)}: q^{-1}(z)\times q^{-1}(z) \rightarrow \matR_{\geq 0}$$
gives to $q^{-1}(z)$ the structure of separable metric space. 
\end{defn}

Metrics along equivariant functions can be enriched
with isometric actions by groupoids. This concept is formalized 
by the following

\begin{defn}
Let $Y$ and $Z$ be standard Borel spaces and $q:Y\rightarrow Z$ be a 
Borel map with a metric $d$ on it.
A \emph{fiberwise isometric action} of a Borel groupoid $\calG$ on 
$q$ is the datum of Borel $\calG$-actions on $Y$ and $Z$ such that $q$ is $\calG$-equivariant and 
$$d_{gz}(gy_1,gy_2)=d_z(y_1,y_2)$$
for every $g\in \calG$, every $z\in Z^{s(g)}$ and every 
$(y_1,y_2)$ in $Y^{s(g)}\times Y^{s(g)}$.
\end{defn}

Fiberwise isometric actions are the key ingredient to 
introduce a fiber version of ergodicity due to Bader and Furman \cite[Definition 2.1]{BF14}.
\begin{defn}\label{definition:relative:isometric:ergodicity}
Let $(\mathcal{G},\nu)$ be a measured groupoid. 
A $\calG$-equivariant Borel map $p:A\rightarrow B$ between Lebesgue $\calG$-spaces 
$(A,\alpha)$ and $(B,\beta)$ is \emph{relatively isometrically ergodic} 
if for every pair of standard Borel $\mathcal{G}$-spaces $Y$ and $Z$, every
fiberwise 
isometric $\calG$-action on $q:Y\rightarrow Z$ and any $\calG$-equivariant commutative square
of $\mathcal{G}$-maps
\begin{center}
  \begin{tikzcd}
  A\arrow{r}{f}\arrow{d}[swap]{p} & Y\arrow{d}{q} \\
  B\arrow{r}{f_0} \arrow[dotted]{ru}{f_1} & Z,
  \end{tikzcd}
  \end{center}
  there exists a $\calG$-equivariant lift $f_1:B\rightarrow Y$.\\
A Lebesgue $\calG$-space $(A,\alpha)$ is \emph{isometrically ergodic} if for every standard Borel $\mathcal{G}$-space $Y$, any fiberwise isometric $\mathcal{G}$-action along $t_Y:Y\rightarrow X$ and any $\mathcal{G}$-map $f:A\rightarrow Y$, there exists a $\mathcal{G}$-equivariant section $f_1:X\rightarrow Y$ such that the following diagram commutes
\begin{center}
  \begin{tikzcd}
  A\arrow{rr}{f}\arrow{ddr}[swap]{t_A} & &Y\arrow[bend right=20]{ldd}[swap]{t_Y} \\& &\\
   &X \arrow[bend right=20,dotted,swap]{ruu}{f_1}&\,.
  \end{tikzcd}
  \end{center}
\end{defn}

\begin{oss}
Being $\calG$-spaces, all the Borel spaces $A,B,Y,Z$ have natural Borel 
projections onto $X$, namely the above square is the basis of an 
inverted pyramid 
\begin{center}
\begin{tikzcd}
A\arrow{rrr}{f}\arrow{rrd}[swap]{p}\arrow[dotted]{rrddd}[swap]{t_A}&  & & Y\arrow{rrd}{q}\arrow[dotted]{lddd}{t_Y} &&
\\
&& B\arrow{rrr}{f_0}\arrow[dotted]{dd}[swap]{t_B} &&& Z\arrow[dotted]{llldd}[]{t_Z}
\\
&&&&&
\\ 
&& X &&&\,
\end{tikzcd}
\end{center}
where all triangles and all squares commute.
However, for sake of notation we will always omit those projections. 
\end{oss}

We state the following facts about (relative) isometric ergodicity that will be useful in Section \ref{section_algebraic_representability}. They are analogous to the ones listed in \cite[Proposition 2.2]{BF14}.

\begin{prop}\label{proposition_properties}
  \begin{itemize}
    \item[(i)] If $(A,\alpha) \rightarrow (B,\beta)$ and $(B,\beta) \rightarrow (C,\gamma)$ are relatively isometrically ergodic, then the composition $A \rightarrow C$ is relatively isometrically ergodic as well.
    \item[(ii)] Given Lebesgue $\mathcal{G}$-spaces $(A,\alpha)$ and $(B,\beta)$, if one of the projections between $p_B:A*B\rightarrow B$ and $p_B:B * A \rightarrow B$ is relatively isometrically ergodic then $(A,\alpha)$ is isometrically ergodic. 
    \item[(iii)] If a Lebesgue $\mathcal{G}$-space $(A,\alpha)$ is isometrically ergodic and $(\mathcal{G},\nu)$ is ergodic, then the semidirect measured groupoid $A\rtimes \mathcal{G}$ is ergodic. 
  \end{itemize}
\end{prop}
\begin{proof} (i). Given a fiberwise isometric action $U \rightarrow V$, suppose we have the following commutative diagram of $\mathcal{G}$-maps
\begin{center}
  \begin{tikzcd}
  A \arrow{rr}\arrow{d} && U \arrow{dd}\\
  B \arrow{rrd} \arrow{d} && \\
  C \arrow{rr}  && V.
  \end{tikzcd}
  \end{center}
It is sufficient to apply twice relative isometric ergodicity: the first time to the map $A \rightarrow B$ and then to the map $B \rightarrow C$.

(ii). We will give a proof assuming that the projection $p_B:A * B \rightarrow B$ is relatively isometrically ergodic, since the other case is analogous. 

We consider a standard Borel $\mathcal{G}$-space $Y$. Given a fiberwise isometric $\calG$-action along the target $t_Y:Y\rightarrow X$ and a $\mathcal{G}$-map $f:A\rightarrow Y$, we need to show that there exists a $\mathcal{G}$-equivariant section $f_1:X\rightarrow Y$ such that $f=f_1 \circ t_A$, where $t_A$ is the target on $A$. 
To this end it is sufficient to consider the fiber product with $B$ on the right. More precisely, if we consider the following commutative diagram
\begin{center}
  \begin{tikzcd}
  A \ast B\arrow{rr}{f \ast \mathrm{id}_B}\arrow{d}[swap]{p_B} && Y \ast B \arrow{d}{p_B} \\
  B\arrow[swap]{rr}{\mathrm{id}_B}  \arrow[dotted]{rru}{f_1} && B,
  \end{tikzcd}
  \end{center}
since $p_B$ is relatively metrically ergodic, we get a function $f_1:B \rightarrow Y \ast B$ telling us that $f$ depends actually only on the target, namely $f(a)=f(t_A(a))$. Thus it is sufficient to define $f_1:X \rightarrow Y$ by $f_1(x):=f(a)$ for $a \in t_A^{-1}(x)$, via the universal property of quotients. 

  (iii). Let $U\subset A$ be a $A\rtimes \mathcal{G}$-invariant Borel subset and consider the map $$f_U:A\rightarrow X\times \{0,1\}\,,\;\;\; f_U\coloneqq t_A\times \mathbbm{1}_U$$
  where $X\times \{0,1\}$ is endowed with the $\mathcal{G}$-action $g(s(g),-)=(t(g),-)$.
  Since $f_U$ is $\mathcal{G}$-equivariant, by isometric ergodicity of $A$ there exists a $\mathcal{G}$-invariant measurable section $X\rightarrow X\times \{0,1\}$ of the projection on the first factor, as shown in the following diagram 
  
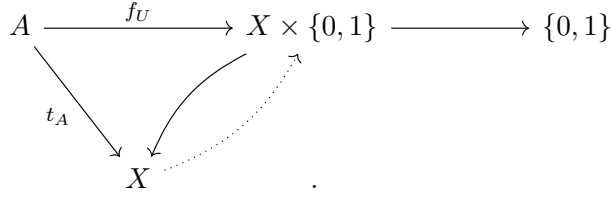
\begin{center}
    \begin{tikzcd}
    A\arrow{rr}{f_U}\arrow{ddr}[swap]{t_A} & &X\times \{0,1\}\arrow[bend right=20]{ldd}[swap]{} \arrow{rr} && \{0,1\} \\
& &\\
     &X \arrow[bend right=20,dotted,swap]{ruu}{}&\,.
    \end{tikzcd}
    \end{center}
Equivalently, the characteristic function $\mathbbm{1}_U$ depends only on the target of an element $a \in A$. By composing the dotted section with the projection on $\{0,1\}$, we obtain a $\mathcal{G}$-invariant measurable map 
  $$\widetilde{f}_U:X\rightarrow \{0,1\} ,$$
which must be essentially constant by the ergodicity of $\mathcal{G}$. This concludes the proof.
\end{proof}


Relative isometric ergodicity is the key concept that we needed to introduce boundary pairs.
(cfr. \cite[Definition 2.3]{BF14})

\begin{defn}\label{definition_G-boundary}
Let $(\calG,\nu)$ be a measured groupoid with unit space $(X,\mu)$. A pair of Lebesgue 
$\calG$-spaces $((\calB_-,\theta_-),(\calB_+,\theta_+))$ such that $(t_{\mathcal{B}_\pm})_*\theta_\pm=\mu$ is a \emph{boundary pair} for $\mathcal{G}$ if
\begin{itemize}
\item[(i)] the semidirect groupoids $\mathcal{B}_\pm \rtimes \calG$ are amenable;
\item[(ii)] the projections $p_\pm:\mathcal{B}_-*\mathcal{B}_+\rightarrow \mathcal{B}_\pm$ 
on the two components are relatively isometrically ergodic.
\end{itemize}
The fiber product $\mathcal{B}_- \ast \mathcal{B}_+$ is done along the target maps $t_{\mathcal{B}_\pm}:\mathcal{B}_\pm \rightarrow X$. 

A Lebesgue $\calG$-space $(\calB,\theta)$ such that $(t_{\mathcal{B}})_*\theta=\mu$ is a \emph{$\mathcal{G}$-boundary} if $((\calB,\theta),(\calB,\theta))$ is a boundary pair for $\mathcal{G}$.

\end{defn}

\begin{oss}
The fact that $(t_{\mathcal{B}_\pm})_*\theta_\pm=\mu$ ensures that $\theta_\pm$ admits a disintegration with respect to $t_{\mathcal{B}_\pm}$ (see \cite[Theorem 2.1]{Hahn}). This assumption will reveal fundamental in the proof of Theorem \ref{theorem_furstenberg_map}.
\end{oss}

\begin{es}
When $\mathcal{G}$ is a locally compact second countable group $\Gamma$ endowed with the Haar measure class,  the notion of boundary pair boils down to the one already given by Bader and Furman \cite[Definition 2.3]{BF14}. For instance, the pair of Poisson boundaries associated to the Markov processes generated by a spread-out probability measure on $\Gamma$ is a boundary pair in the sense of Definition \ref{definition_G-boundary} by \cite[Theorem 2.7]{BF14}. 
\end{es}

The first non-trivial example of
Definition \ref{definition_G-boundary} is given for a semidirect groupoid associated to a probability measure-preserving action.
In such context a boundary pair boils down the product between the boundary pair of the group and the unit space. 

\begin{rec_thm}[\ref{theorem_semidirect}]
   Let $\Gamma$ be a locally compact second countable group and consider
    a probability measure preserving action $\Gamma\curvearrowright X$ on a
    standard Borel probability space $(X,\mu)$. Let $\mathcal{G}=\Gamma \ltimes X$ be the semidirect groupoid determined by the $\Gamma$-action on $X$. 
    If $((B_-,\beta_-),(B_+,\beta_+))$ is a boundary pair for $\Gamma$,
    then $(B_-\times X,B_+\times X)$ endowed with the product measures
    $(\theta_-,\theta_+)\coloneqq (\beta_-\otimes \mu,\beta_+\otimes \mu)$ is a boundary pair for $\mathcal{G}$. 

\noindent In particular, if $(B,\beta)$ is a $\Gamma$-boundary, then $(B \times X, \beta \otimes \mu)$ is a $\mathcal{G}$-boundary. 
\end{rec_thm}
\begin{proof}
The fact that the direct images of $(\theta_-,\theta_+)$ through the target maps coincide with $\mu$ is trivial because they are defined as products.

For the remaining part of the proof we will follow the line adopted in \cite[Theorem 1]{sarti:savini:3}.
By assumption, the actions $\Gamma \curvearrowright B_\pm$ are amenable and the 
projections $B_-\times B_+\rightarrow B_\pm$ are relatively isometrically ergodic.

First of all, the amenability of $\mathcal{B}_\pm \rtimes (\Gamma\ltimes X)$ follows immediately
by the amenability of $\Gamma \curvearrowright B_\pm$ \cite[Proposition 4.3.4]{zimmer:libro}, thanks to the groupoid isomorphism
\begin{gather*}
  ((B_\pm\times X) \rtimes (\Gamma\ltimes X),( (\beta_\pm\otimes \mu) \circ (m_{\Gamma}\otimes \mu))\rightarrow ((\Gamma\ltimes(B_\pm\times X),(\beta_\pm\otimes \mu)\circ m_{\Gamma}),\\
  ((b_\pm,x),(\gamma,x))\mapsto (\gamma,(b_\pm,x)).
\end{gather*}
where $m_\Gamma$ is the Haar measure on $\Gamma$ and $\circ $ defines the usual integration of a Borel system with respect to the measure 
on the unit space.

It remains to show relative isometric ergodicity of the projections
$$p_\pm:(B_-\times X)*(B_+\times X)\rightarrow B_\pm\times X.$$
Since the proofs of the two cases are analogous, we focus on the first term projection 
$p_-:(B_-\times X)*(B_+\times X)\rightarrow B_-\times X$.
To this end,
we consider a fiberwise isometric $\mathcal{G}$-action $q:Y\rightarrow Z$ between Borel
$\mathcal{G}$-spaces $Y$ and $Z$ and the commutative diagram
\begin{center}
  \begin{tikzcd}
    (B_-\times X)*(B_+\times X)\arrow{r}{f}\arrow{d}[swap]{p_-} & Y\arrow{d}{q} \\
    B_-\times X\arrow{r}{f_0} & Z,
  \end{tikzcd}
  \end{center}
  where both $f:(B_-\times X)*(B_+\times X)\rightarrow Y$ and $f_0:B_-\times X\rightarrow Z$ are 
  $\mathcal{G}$-maps. 

Thanks to the $\mathcal{G}$-isomorphism of measure spaces
\begin{equation}\label{equation_isomorphism1}
  (B_-\times X)*(B_+\times X) \rightarrow (B_-\times B_+)\times X,\;\;\;
((b_-,x),(b_+,x))\mapsto ((b_-,b_+),x) \,
\end{equation}
the above diagram boils down to 
\begin{equation}\label{equation_diagram1}
  \begin{tikzcd}
    (B_-\times B_+)\times X\arrow{r}{f}\arrow{d}[swap]{p_-} & Y\arrow{d}{q} \\
    B_-\times X\arrow{r}{f_0} & Z,
  \end{tikzcd}
\end{equation}
where we abusively denote again by $p_-$ and $f$ 
the composition of $p_-$ and $f$ with the isomorphism of Equation
 \eqref{equation_isomorphism1}, respectively.

We denote by $\mathrm{L}^0(X,Y)$ the space of all measurable functions between $X$ and $Y$ identified if they coincide $\mu$-almost everywhere. We adopt a similar definition for $\mathrm{L}^0(X,Z)$. With a slight abuse of notation, we will always pick a representative to refer to an equivalence class of sections lying in $\mathrm{L}^0(X, \ \cdot \ )$. 

The space $\mathrm{L}^0(X,Y)$ is a Polish space with the structure coming from the topology of convergence in measure \cite[Section 4.4]{Wheeden},\cite[Notation 2.4]{fisher:morris:whyte}. A generating family of Borel sets \cite[Remark 2.5]{fisher:morris:whyte} for the associated $\sigma$-algebra is given by 
\begin{equation}\label{eq generating Borel structure}
\Delta_{X_0,Y_0,\varepsilon}:=\{ f \ | \ \mu(X_0 \cap f^{-1}(Y_0) ) < \varepsilon \},
\end{equation}
where $X_0 \subset X, Y_0 \subset Y$ are Borel subsets and $\varepsilon >0$.

We set $\widetilde{q}:\mathrm{L}^0(X,Y)\rightarrow \mathrm{L}^0(X,Z)$ as $\widetilde{q}(\varphi)\coloneqq q\circ \varphi$ 
and define the distance $\widetilde{d}_{\psi}$ on the fiber $\widetilde{q}^{-1}(\psi)$ as 
  \begin{equation}\label{equation_metric}
    \widetilde{d}_{\psi}(\varphi_1,\varphi_2)\coloneqq\int\limits_X \frac{d_{\psi(x)}(\varphi_1(x),\varphi_2(x))}{1+d_{\psi(x)}(\varphi_1(x),\varphi_2(x))}d\mu(x),
  \end{equation}
where $d_{\psi(x)}$ is the distance given along $q:Y \rightarrow Z$ (and the latter makes sense because $q(\varphi_1(x))=q(\varphi_2(x))=\psi(x)$). Notice that the measurability of $q$, applied to the Borel sets of Equation \eqref{eq generating Borel structure}, implies that $\widetilde{q}$ is measurable.

The target map $t_Y$ induces a Borel map $\widetilde{t}_Y:\mathrm{L}^0(X,Y) \rightarrow \mathrm{L}^0(X,X)$ (and the same does $t_Z$). We define the preimage of the identity $\mathrm{id}_X$ as
$$
\mathcal{S}(X,Y):=\widetilde{t}_Y^{-1}(\mathrm{id}_X),
$$
which is the subspace of (classes of) measurable sections to $t_Y$. The latter is standard Borel, being a Borel subset of a Polish space. The same holds for $\mathcal{S}(X,Z)$. Moreover we can restrict $\widetilde{q}:\mathcal{S}(X,Y) \rightarrow \mathcal{S}(X,Z)$ getting a well-defined Borel map.

If we denote by $\ \cdot_Y \ $ the $(\Gamma \ltimes X)$-action on $Y$ and by $\ \cdot_Z \ $ the one on $Z$, the map
$$\Gamma \times \mathcal{S}(X,Y)\rightarrow \mathcal{S}(X,Y)\,,\;\;\; 
\gamma\cdot \varphi(x)\coloneqq (\gamma,\gamma^{-1}x)\cdot_Y \varphi(\gamma^{-1} x)$$
is a fiberwise isometric $\Gamma$-action on $\widetilde{q}$.
In fact, we have that 
\begin{align*}
  &\widetilde{d}_{\gamma\cdot \psi}(\gamma\cdot \varphi_1,\gamma\cdot \varphi_2)=\\
 & = \int\limits_X \frac{ d_{(\gamma,\gamma^{-1}x)\cdot_Z \psi(\gamma^{-1} x)}((\gamma,\gamma^{-1}x) \cdot_Y\varphi_1(\gamma^{-1}x),(\gamma,\gamma^{-1}x) \cdot_Y\varphi_2(\gamma^{-1}x))}{1+d_{(\gamma,\gamma^{-1}x)\cdot_Z \psi(\gamma^{-1} x)}((\gamma,\gamma^{-1}x) \cdot_Y\varphi_1(\gamma^{-1}x),(\gamma,\gamma^{-1}x) \cdot_Y\varphi_2(\gamma^{-1}x))}d\mu(x)\\
 & = \int\limits_X \frac{ d_{ \psi(\gamma^{-1} x)}(\varphi_1(\gamma^{-1}x),\varphi_2(\gamma^{-1}x))}{1+d_{ \psi(\gamma^{-1} x)}(\varphi_1(\gamma^{-1}x),\varphi_2(\gamma^{-1}x))}d\mu(x)\\
 & = \int\limits_X \frac{d_{\psi(y)}(\varphi_1(y),\varphi_2(y))}{1+d_{\psi(y)}(\varphi_1(y),\varphi_2(y))}d\mu(y)= \widetilde{d}_{\psi}(\varphi_1,\varphi_2)
\end{align*}
where we moved from the second line to the third one exploiting the fact that $q$ is a fiberwise isometric $\mathcal{G}$-action
and from the third line to the last one thanks to the $\Gamma$-invariance of $\mu$.
Moreover defining 
 $\widetilde{f}:B_-\times B_+\rightarrow \mathcal{S}(X,Y)$ as $$\widetilde{f}(b_-,b_+)(x)\coloneqq f((b_-,b_+),x)$$
 and similarly $\widetilde{f_0}:B_-\rightarrow \mathcal{S}(X,Z)$ as 
 $$\widetilde{f_0}(b_-)(x)\coloneqq f_0(b_-,x),$$
the diagram of Equation \eqref{equation_diagram1} becomes
\begin{equation}\label{equation_diagram2}
  \begin{tikzcd}
    B_-\times B_+\arrow{r}{\widetilde{f}}\arrow{d}[swap]{p_1} & \mathcal{S}(X,Y)\arrow{d}{\widetilde{q}} \\
    B_-\arrow{r}{\widetilde{f_0}} & \mathcal{S}(X,Z).
  \end{tikzcd}
\end{equation}

Since the first factor projection $p_1:B_-\times B_+\rightarrow B_-$ is relatively isometrically
ergodic, there exists a lift $\widetilde{F}:B_-\rightarrow \mathcal{S}(X,Y)$ of $\widetilde{f_0}$ 
such that the following diagram commutes
\begin{equation}\label{equation_diagram3}
  \begin{tikzcd}
    B_-\times B_+\arrow{r}{\widetilde{f}}\arrow{d}[swap]{p_1} & \mathcal{S}(X,Y)\arrow{d}{\widetilde{q}} \\
    B_-\arrow{r}{\widetilde{f_0}}\arrow[dotted]{ru}{\widetilde{F}} & \mathcal{S}(X,Z).
  \end{tikzcd}
\end{equation}
Then, setting 
$$F:B_-\times X\rightarrow Y\,,\;\;\; F(b_-,x)\coloneqq \widetilde{F}(b_-)(x)$$
we obtain a lift of $f_0$ fitting in the following commutative diagram
\begin{equation}\label{equation_diagram4}
    \begin{tikzcd}
      (B_-\times B_+)\times X\arrow{r}{f}\arrow{d}[swap]{p_-} & Y\arrow{d}{q} \\
      B_-\times X\arrow{r}{f_0}\arrow[dotted]{ru}{F} & Z.
    \end{tikzcd}   
  \end{equation}
 This shows the relative isometric ergodicity of $p_-$ and concludes the proof.
\end{proof}

\begin{oss}
The previous theorem is relevant also from a cohomological point of view. Let $\calG=\Gamma \ltimes X$ be the semidirect groupoid associated to a measure preserving action of a locally compact group. Given a coefficient $\calG$-module (see \cite[Definition 1.2.1]{monod:libro}), in virtue of the \emph{exponential law} \cite[Theorem 1]{sarti:savini:23} we know that the measurable bounded cohomology of $\calG$ with $E$-coefficients is isomorphic to the continuous bounded cohomology of $\Gamma$ with twisted coefficients $\mathrm{L}^\infty_{\mathrm{w}^\ast}(X,E)$, namely
$$
\mathrm{H}^k_{mb}(\calG,E) \cong \mathrm{H}^k_{cb}(\Gamma,\mathrm{L}^\infty_{\mathrm{w}^\ast}(X,E)),
$$
for every $k \geq 0$. If $(B_-,B_+)$ is a boundary pair for $\Gamma$ in the sense of Bader-Furman, the amenability of $B_\pm$ guarantees that the complex of essentially bounded functions on $B_\pm$ computes the bounded cohomology of $\Gamma$ with the same coefficients \cite[Theorem 1]{burger2:articolo}. As a consequence, we have that 
\begin{equation}\label{eq:isomorphism:cohomology}
\mathrm{H}^k_{cb}(\Gamma,\mathrm{L}^\infty_{\mathrm{w}^\ast}(X,E)) \cong \mathrm{H}^k(\mathrm{L}^\infty_{\mathrm{w}^\ast}(B_\pm^{\bullet+1},\mathrm{L}^\infty_{\mathrm{w}^\ast}(X,E))^\Gamma),
\end{equation}
where the $\Gamma$-action on $\mathrm{L}^\infty_{\mathrm{w}^\ast}(X,E)$ is induced by the one of $\calG$ on $E$. 

By Theorem \ref{theorem_semidirect} we know that $\mathcal{B}_\pm=B_\pm \times X$ is a boundary pair of the groupoid $\calG$. The usual exponential law for essentially bounded functions \cite[Corollary 2.3.3]{monod:libro} can be exploited to show that the complex on the right-hand side of Equation \eqref{eq:isomorphism:cohomology} is actually isomorphic to
$$
\mathrm{L}^\infty_{\mathrm{w}^\ast}(\mathcal{B}_\pm^{(\bullet+1)},E)^\calG.
$$
Here $\mathcal{B}_\pm^{(k)}$ is the $k$-fiber product of $\mathcal{B}_\pm$ with respect to the target map and the apex $\calG$ refers to the $\calG$-invariants functions, defined as in \cite[Section 4.2]{sarti:savini:23}. To sum up the above chain of isomorphisms, we obtain that 
$$
\mathrm{H}^k_{mb}(\calG,E) \cong \mathrm{H}^k(\mathrm{L}^\infty_{\mathrm{w}^\ast}(\mathcal{B}_\pm^{(\bullet+1)},E)^\calG).
$$
Thus, in the case of an action groupoid $\calG$, the boundary pair can be used to compute the measurable bounded cohomology of $\calG$. In the future, we hope to prove that this is actually true for any measured groupoid, generalizing in this way the result by Burger and Monod. 
\end{oss}

%
%

\section{The Poisson boundary} \label{sec:poisson}

In this section we are going to focus our attention on the notion of Poisson boundary for an invariant Markov operator on a measured groupoid. After a brief introduction about the general setting, we describe the particular case of a Markov operator generated by a probability measure $\pi$ lying in the same class of the integrated Haar system. We are going to show that the Poisson boundaries associated to the measure $\pi$ and its reflected variant $\check{\pi}$, respectively, determine a boundary pair in the sense of Definition \ref{definition_G-boundary}. For the construction of the Poisson boundary we mainly refer to Kaimanovich \cite{Kai05} and the references therein. 

\subsection{The Poisson boundary of an invariant Markov process on a measured groupoid}\label{subsec:markov}

Let $\mathcal{G}$ be a Borel groupoid with units $X=\mathcal{G}^{(0)}$. We consider a Haar system $\lambda=\{ \lambda^x \}_{x \in X}$ with respect to the target map and we denote by $\mu$ a quasi-invariant probability measure on $X$ with respect to $\lambda$. A \emph{Markov chain} defined on $\mathcal{G}$ is the assignment of a Borel probability measure $\pi^g$ on $\mathcal{G}$ to any element $g \in \mathcal{G}$ so that the function
$$
\pi(f): \mathcal{G} \rightarrow \mathbb{R}, \ \ \pi(f)(g):=\int_{\mathcal{G}} f(h)d\pi^g(h)
$$
is Borel for every non-negative Borel function $f$ on $\mathcal{G}$. We say that the Markov chain is $\mathcal{G}$-\emph{invariant} if it holds
\begin{equation}\label{eq:ginvariance}
\pi^{gh}=g_\ast \pi^h , 
\end{equation}
for every pair of composable elements $(g,h) \in \mathcal{G}^{[2]}$. As noticed by Kaimanovich \cite[Section 3B]{Kai05}, Equation \eqref{eq:ginvariance} implies that $\pi^g(\mathcal{G}^{t(g)})=1$
for any $g\in \mathcal{G}$. By definition, every $\mathcal{G}$-invariant Markov chain $\{ \pi^g \}_{g \in \mathcal{G}}$ is uniquely determined by the transition probabilities defined on the units, namely $\{ \pi^x \}_{x \in X}$. Conversely, any Borel system of probability measures supported on the fibers $\{ \mathcal{G}^x \}_{x \in X}$ of the target map can be uniquely extended to a $\mathcal{G}$-invariant Markov chain \cite[Proposition 3.4]{Kai05} by setting
$$
\pi^g:=g_\ast \pi^{s(g)},
$$ 
for every $g \in \mathcal{G}$. 

Given a Markov chain on $\mathcal{G}$, we can naturally define the associated \emph{Markov operator} on the space of bounded Borel functions on $\mathcal{G}$ by setting
$$
(P\varphi)(g):=\int_{\mathcal{G}} \varphi(h)d\pi^g(h) ,
$$
where $\varphi$ is a bounded Borel function. We say that $\varphi$ is $P$-\emph{harmonic} if it holds $\varphi=P\varphi$.

 When the chain is $\mathcal{G}$-invariant, the Markov operator $P$ boils down to a $\mathcal{G}$-invariant collection of Markov operators $\{ P^x \}_{x \in X}$ on the fibers $\{\mathcal{G}^x \}_{x \in X}$ of the target map \cite[Section 3.B]{Kai05}. More precisely, the collection of measures defining $P^x$ is given by $\{ \pi^g \}_{t(g)=x}$. 

By duality, the \emph{dual Markov operator} acts on the space of non-negative Borel measures on $\mathcal{G}$ as follows
$$
m P:=\int_{\mathcal{G}} \pi^g dm(g) , 
$$
where $m$ is a non-negative measure. We say that $m$ is $P$-\emph{invariant}, $P$-\emph{quasi-invariant} or $P$-\emph{adapted} if $m P=m$, $m P$ is equivalent to $m$ or $m P$ is absolutely continuous with respect to $m$, respectively. For our purpose, we are going  to assume that 
$$
\mu \circ \lambda=\int_X \lambda^x d\mu(x) 
$$
is $P$-adapted. This condition is equivalent to assume that the measure $\lambda^x$ is adapted with respect to the fiberwise Markov operator $P^x$, for almost every $x \in X$. The measure $\mu \circ \lambda$ is $P$-adapted when we consider \emph{absolutely continuous} transition probabilities $\{ \pi^g \}_{g \in \mathcal{G}}$, namely when it holds
\begin{equation}\label{eq:abscont}
\pi^g  \prec \lambda^{t(g)}
\end{equation}
for every $g \in \mathcal{G}$. Again by $\mathcal{G}$-invariance, it is sufficient to verify Equation \eqref{eq:abscont} only on the space of units.

Let now consider the infinite fiber product
$$
\Omega_{\mathcal{G}}:=\{ \omega=(g_0,g_1,g_2,\ldots) \in \mathcal{G}^{\mathbb{N}} \ | \ t(g_i)=t(g_j) \ \ \forall i,j \in \mathbb{N} \},
$$
where we call an element $\omega \in \Omega_{\mathcal{G}}$ a \emph{forward trajectory}. Since the target is constant along the components of a forward trajectory, we have a natural target map 
$$
t_\Omega:\Omega_{\mathcal{G}} \rightarrow X, \ \ t_\Omega(\omega)=t(g_0).
$$
Given $g \in \mathcal{G}$, we can define a Borel probability measure $\mathbb{P}_g$ such that, for any bounded Borel functions $\varphi_0,\ldots,\varphi_n$ on $\mathcal{G}$, it holds that
\begin{equation}\label{eq:implicit:def}
\int_{\Omega_{\mathcal{G}}} \varphi_0(g_0) \ldots \varphi_n(g_n) d\mathbb{P}_g(\omega)=(\varphi_0 \cdot P( \ldots (\varphi_{n-1} \cdot (P\varphi_n)) \ldots ))(g),
\end{equation}
which can be alternatively rewritten as
\begin{align}\label{eq:explicit:def}
&\int_{\Omega_{\mathcal{G}}} \varphi_0(g_0) \ldots \varphi_n(g_n) d\mathbb{P}_g(\omega)\\
=&\int_{\mathcal{G}}\ldots\int_{\mathcal{G}}\varphi_0(g_0) \ldots \varphi_n(g_n) d\pi^{g_{n-1}}(g_n)\ldots d\pi^{g_0}(g_1)d\delta_g(g_0). \nonumber
\end{align}
Notice that $\mathbb{P}_g$ is well-defined on $\Omega_{\calG}$ since the right-hand side of Equation \eqref{eq:explicit:def} forces all the $g_i$'s to have the same target. Following Benoist and Quint \cite[Chapter 2.1.1]{BQ16} we say that $\mathbb{P}_g$ is the $g$-\emph{Markov measure} on the space of forward trajectories associated to the Markov operator $P$. More generally, given a probability measure $m$ on $\mathcal{G}$ as initial distribution, we can consider the \emph{Markov measure associated to} $m$ by applying the usual integration method, namely
\begin{equation}\label{equation_measure_forward_traj}
  \mathbb{P}_m:=\int_{\mathcal{G}} \mathbb{P}_g dm(g).  
\end{equation}
In this way we immediately see that $\mathbb{P}_g$ is the Markov measure with initial distribution $\delta_g$, the Dirac mass at $g$. Notice that the $n$-th one-dimensional distribution associated to the Markov measure $\mathbb{P}_m$ is precisely $m P^n$, namely the Markov operator applied $n$ times to $m$. 

We define the \emph{time shift} by
\begin{align*}
T:\Omega_{\mathcal{G}} &\longrightarrow \Omega_{\mathcal{G}}\\
\omega=(g_0,g_1,g_2,\ldots) &\mapsto T\omega:=(g_1,g_2,g_3,\ldots). 
\end{align*}
It is worth noticing that, given an initial distribution $m$ on $\mathcal{G}$, the time shift acts naturally on the Markov measure $\mathbb{P}_m$ as follows
\begin{equation}
T_\ast \mathbb{P}_m=\mathbb{P}_{mP}.
\end{equation}
Given two forward trajectories $\omega,\omega' \in \Omega_{\mathcal{G}}$, we say that they are \emph{stably} $T$-\emph{equivalent}, if there exist $k,\ell \in \mathbb{N}$ such that $T^k \omega=T^\ell \omega'$. By choosing as $m$ the integrated Haar system $\mu \circ \lambda$, there must exist a unique (up to isomorphism) measurable space $\mathcal{B}$ and a Borel map 
$$
\mathbf{bnd}:\Omega_{\mathcal{G}} \rightarrow \mathcal{B},
$$
such that the $\sigma$-algebra generated by $\mathbf{bnd}$ coincides with the one generated by stable $T$-equivalence \cite[Section 5.A]{Kai05}. We say that $\mathcal{B}$ is the space of ergodic components with respect to the time shift and we endow it with the measure class $[\theta_{\mu \circ \lambda}]:=\mathbf{bnd}_\ast[\mathbb{P}_{\mu \circ \lambda}]$.

\begin{defn}\label{def:poisson:boundary}
Let $(\mathcal{G},\mu \circ \lambda)$ be a measured groupoid with Borel Haar system $\lambda$ and measure $\mu$ on the unit space. Consider a $\mathcal{G}$-invariant Markov chain $\{ \pi^g \}_{g \in \mathcal{G}}$ with associated Markov operator $P$. Suppose that $\mu \circ \lambda$ is $P$-adapted. The \emph{Poisson boundary} of the Markov operator $P$ is the space $\mathcal{B}$ of ergodic components of $(\Omega_{\mathcal{G}},\mathbb{P}_{\mu \circ \lambda})$ with respect to $T$ endowed with the measure class $[\theta_{\mu \circ \lambda}]$.  
\end{defn}

We first observe that the target map $t_\Omega$ on $\Omega_{\mathcal{G}}$ naturally descends to a target map on $\mathcal{B}$, namely
$$
t_{\mathcal{B}}:\mathcal{B} \rightarrow X, \ \ t_{\mathcal{B}}([\omega])=t_\Omega(\omega).  
$$
As noticed by Kaimanovich, the generic fiber $t_{\mathcal{B}}^{-1}(x)=\mathcal{B}^x$ of the above map is given by the Poisson boundary of the fiberwise Markov operator $P^x$ defined on $\mathcal{G}^x$ \cite[Section 5.B]{Kai05}. 

There exists a natural left action of $\mathcal{G}$ on $\Omega$ given by 
$$
\mathcal{G} \ast \Omega_\mathcal{G} \rightarrow \Omega_\mathcal{G}, \ \ (g,\omega=(g_0,g_1,\ldots)) \mapsto (gg_0,gg_1,\ldots),
$$
where $s(g)=t(\omega)$. Since the $\mathcal{G}$-action commutes with the time shift $T$, it naturally descends to an action on the boundary $\mathcal{B}$, that is 
$$
\mathcal{G} \ast \mathcal{B} \rightarrow \mathcal{B}, \ \ (g, b=[\omega]) \mapsto gb:=[g\omega]. 
$$
Additionally, since $\mu \circ \lambda$ is quasi-symmetric, the Markov measure $\mathbb{P}_{\mu \circ \lambda}$ is quasi-symmetric for the groupoid $\Omega_{\calG} \rtimes \mathcal{G}$, and hence the measure class $[\theta_{\mu \circ \lambda}]$ is  symmetric with respect to the induced $\mathcal{G}$-action on $\mathcal{B}$ \cite[Section 5.B]{Kai05}.  


\subsection{The boundary pair generated by a quasi-symmetric left quasi-invariant probability measure.}\label{subsec:gboundary}

In this section we show that the Poisson boundaries associated to the invariant Markov operators determined by a probability measure in the class of the integrated Haar system provide an example of boundary pairs in the sense of Definition \ref{definition_G-boundary}. 

We fix the following setting.
Let $(\mathcal{G},\mu \circ \lambda)$ be a measured groupoid, where $\lambda=\{\lambda^x\}_{x \in X}$ is a Borel Haar system and $\mu$ is a quasi-invariant probability measure on $X$ with respect to $\lambda$.
Let $\pi$ be a probability measure equivalent to $\mu \circ \lambda$ such that $t_\ast \pi=s_\ast \pi=\mu$, where $t$ and $s$ are the target and the source on $\calG$, respectively.
Denote by $\{\pi^x\}_{x\in X}$ the $t$-disintegration of $\pi$ and define the associated $\calG$-invariant Markov chain $\{\pi^g\}_{g\in \mathcal{G}}$, obtained through the left $\mathcal{G}$-action (see Section \ref{subsec:markov}). By construction, the left quasi-invariance of $\pi$ implies that $$\pi^g \sim \lambda^{t(g)},$$ thus we have absolutely continuous transition probabilities. 

We consider the \emph{reflected measure} $\check{\pi}$, namely the direct image on $\calG$ of the measure $\pi$ via the inverse map. Following what we did for $\pi$, we can consider the $t$-disintegration $\{\check{\pi}^x\}_{x\in X}$ and then extend it to a $\calG$-invariant Markov chain $\{\check{\pi}^g\}_{g\in \calG}$.
We denote by $\mathbb{P}_\mu$ the Markov measure on $\Omega_{\calG}$ defined by the invariant Markov chain determined by $\pi$ with initial distribution $\mu$. We define $\check{\mathbb{P}}_\mu$ in a similar way by looking at the Markov chain generated by $\check{\pi}$. The Poisson boundaries associated to the above Markov measures will be denoted by $(\mathcal{B},\theta_{\mu})$ and $(\check{\mathcal{B}},\check{\theta}_{\mu})$. 

The main goal of this section is to prove the following
\begin{rec_thm}[\ref{theorem_poisson_boundary}]
Let $(\calG,\nu)$ be a measured groupoid with unit space $(X,\mu)$ and let $\pi$ be a probability measure equivalent to $\nu$ such that $t_*\pi=s_*\pi=\mu$, where $t$ and $s$ are the target and the source on $\mathcal{G}$, respectively. Denote by $\check{\pi}$ the direct image of $\pi$ under the inverse map. The pair $((\mathcal{B},\theta_{\mu}),(\check{\mathcal{B}},\check{\theta}_{\mu}))$ of Poisson boundaries associated to the Markov operators generated by $\pi$ and $\check{\pi}$, respectively, with initial distribution $\mu$ is a boundary pair for $\mathcal{G}$.

\noindent In particular, if $\pi$ is symmetric, namely $\pi=\check{\pi}$, then $(\mathcal{B},\theta_{\mu})$ is a $\mathcal{G}$-boundary.
\end{rec_thm}

This extends the analogous result given by Bader and Furman \cite[Theorem 2.7]{BF14} for groups. Although the proof is a suitable adaptation to the groupoids framework of the one by Bader and Furman, some new technicalities arise.

The first necessary step is to show an invariance property of the measure $\mathbb{P}_\mu$ that we will explicitly use later on.

\begin{lemma}\label{lemma_preserve}
  Retain the notation given at the beginning of the section. Then the map 
  \begin{equation}\label{equation_map_trajectories}
    S:\Omega_{\mathcal{G}} \longrightarrow \Omega_{\mathcal{G}}, \ \ S(g_0,g_1,g_2,\ldots)\coloneqq (g_1^{-1}g_1,g_1^{-1}g_2,\ldots)
  \end{equation}
  preserves the measure $\mathbb{P}_{\mu}$, namely
  \begin{equation*}
  S_\ast(\mathbb{P}_\mu)=\mathbb{P}_\mu.
  \end{equation*} 
\end{lemma}
\begin{proof}
We start considering the fiber product 
$$
\calG \ast \Omega_{\calG}:=\{ (g, \omega=(g_0,g_1,g_2,\ldots)) \in \calG \times \Omega_{\calG} \ | \ g=g_0 \},
$$
endowed with the fiber product measure
$$
\pi \widehat{\otimes} \mathbb{P}_\pi:=\int_{\calG} (\delta_g \otimes \mathbb{P}_g) d\pi(g). 
$$
There is a natural map 
$$
\sigma:\calG \ast \Omega_{\calG} \rightarrow \Omega_{\calG}, \ \ \ \sigma(g,\omega)=g^{-1}\omega,
$$
where the multiplication by $g^{-1}$ is done diagonally. If $\omega=(g_0,g_1,g_2,\ldots)$ and $T$ is the time shift on $\Omega_{\calG}$, we can see that the map $S$ can be rewritten as
\begin{equation}\label{eq:sigma:dis}
S\omega=\sigma(g_1,T\omega)=g_1^{-1}T\omega. 
\end{equation}
Recalling that the time shift acts on $\mathbb{P}_\mu$ as follows
$$
T_\ast \mathbb{P}_\mu=\mathbb{P}_{\mu P}=\mathbb{P}_\pi, 
$$
Equation \eqref{eq:sigma:dis} implies that 
$$
S_\ast(\mathbb{P}_\mu)=\sigma_\ast(\pi \widehat{\otimes} \mathbb{P}_\pi). 
$$
As a consequence, it is sufficient to show $\sigma_\ast(\pi \widehat{\otimes} \mathbb{P}_\pi)=\mathbb{P}_\mu$. Let $\varphi$ be a bounded Borel function on $\Omega_{\calG}$. 
\begin{align*}
\langle \sigma_\ast(\pi \widehat{\otimes} \mathbb{P}_\pi), \varphi \rangle &=\langle\pi \widehat{\otimes} \mathbb{P}_\pi, \varphi \circ \sigma \rangle\\
&=\int_{\calG} \int_{\Omega_{\calG}} \int_{\mathcal{G}}\varphi(\sigma(h,\omega))d\delta_g(h)d\mathbb{P}_g(\omega)d\pi(g)\\
&=\int_{\calG} \int_{\Omega_{\calG}} \varphi(g^{-1}\omega)d\mathbb{P}_g(\omega)d\pi(g)\\
&=\int_{\calG} \int_{\Omega_{\calG}} \varphi(\omega)d\mathbb{P}_{s(g)}(\omega)d\pi(g)\\
&=\int_X \int_{\Omega_{\calG}} \varphi(\omega)d\mathbb{P}_x(\omega)d\mu(x)=\langle \mathbb{P}_\mu,\varphi \rangle. 
\end{align*}
In the above computation we passed from the second line to the third one by using the definition of $\sigma$, we moved from the third line to the fourth one thanks to the fact that $g^{-1}_\ast \mathbb{P}_g=\mathbb{P}_{s(g)}$ and we concluded exploiting the definition of $\mathbb{P}_\mu$ together with the hypothesis $s_\ast(\pi)=\mu$. This concludes the proof. 
\end{proof}

We move on in our investigation by showing that, under the assumptions settled at the beginning of this section, the Poisson boundaries we are dealing with are amenable $\mathcal{G}$-space. Here the crucial assumption is that both $\pi$ and $\check{\pi}$ are equivalent to $\mu\circ \lambda$.
\begin{lemma}\label{lem:amenability}
Let $(\mathcal{G},\mu \circ \lambda)$ be a measured groupoid and let $\pi$ a probability measure equivalent to $\mu \circ \lambda$. Consider the Poisson boundaries $(\mathcal{B},\theta_{\mu})$ and $(\check{\mathcal{B}},\check{\theta}_{\mu})$ associated to the Markov operators generated by $\pi$ and $\check{\pi}$, respectively, both with initial distribution $\mu$. Then, the semidirect groupoids $(\mathcal{B} \rtimes \mathcal{G}, \theta_\mu \circ \lambda)$ and $(\check{\mathcal{B}} \rtimes \calG,  \check{\theta}_\mu \circ \lambda)$ are amenable.  
\end{lemma}

\begin{proof}
We are going to do the proof only in the case of $(\mathcal{B},\theta_\mu)$, since the argument for $(\check{\mathcal{B}},\check{\theta}_\mu)$ is analogous.

Kaimanovich \cite[Theorem 5.2]{Kai05} proved the amenability in the general case of any Poisson bundle associated to an invariant Markov operator with starting distribution $\mu \circ \lambda$. Thus it is sufficient to show that $\theta_\mu \in [\theta_{\mu \circ \lambda}]$. The latter condition is a consequence of the fact that $\theta_{\mu \circ \lambda}$ is equivalent to $\theta_\pi$. Indeed, we have that 
$$
T_\ast \mathbb{P}_\mu=\mathbb{P}_{\mu P}=\mathbb{P}_\pi,
$$
and hence
\begin{equation}\label{eq:equivalence:thetamu}
\theta_\mu=\mathbf{bnd}_\ast(\mathbb{P}_\mu)=\mathbf{bnd}_\ast(T_*\mathbb{P}_\mu)=\mathbf{bnd}_\ast(\mathbb{P}_\pi)=\theta_\pi.
\end{equation}
This shows the statement and concludes the proof. 
\end{proof}

Notice that Equation \eqref{eq:equivalence:thetamu} holds also fiberwise, namely 
\begin{equation*}
  \theta^x=\theta_{\pi^x} \hspace*{1 cm} \text{and} \hspace*{1 cm}\check{\theta}^x=\check{\theta}_{\check{\pi}^x} \,,
\end{equation*}
where $\theta^x=\mathbf{bnd}_\ast(\mathbb{P}_x)$ and $\check{\theta}^x=\check{\mathbf{bnd}}_\ast(\check{\mathbb{P}}_x)$. Here we are denoting by $\check{\mathbf{bnd}}:\Omega_{\calG} \rightarrow \check{\calB}$ the boundary projection. By the absolute continuity of the transition probabilities \cite[Section 5.A]{Kai05}, we can write a fiberwise Poisson formula. For the ease of the reader, we will focus only on the case we are going to use later in the proof. We denote by $P^x$ the fiberwise Markov operator on $\calG^x$ determined by the chain $\{\pi^g\}_{t(g)=x}$. The \emph{fiberwise Poisson formula} explicitly realizes the isomorphism between the space $\mathrm{H}^\infty(\calG^x,\lambda^x,P^x)$ of bounded $P^x$-harmonic functions identified modulo $\lambda^x$ and the space $\mathrm{L}^\infty(\calB^x,\theta^x)$ of essentially bounded functions on $\calB^x=t_{\calB}^{-1}(x)$:
\begin{equation}\label{eq:poisson:fiber:inverse}
  \mathcal{P}^x: \mathrm{L}^\infty(\calB^x,\theta^x) \rightarrow \mathrm{H}^\infty(\mathcal{G}^x,\lambda^x,P^x),  \ \ \ (\mathcal{P}^x\varphi)(g):=\int_{\calB^x} \varphi(b)d\theta^g(b)\,.
  \end{equation}

We want to show that the boundary $(\check{\mathcal{B}},\check{\theta}_\mu)$ satisfies a stationarity which generalizes the one already known for boundaries of groups. A similar result holds also for $(\mathcal{B},\theta_\mu)$, but since we are not going to use it, we prefer to omit the computation. With the help of the target $t_{\check{\mathcal{B}}}:\check{\calB} \rightarrow X$, we can disintegrate $\check{\theta}_\mu$ as follows
$$
\check{\theta}_\mu=\int_X \check{\theta}^x d\mu(x).
$$
Using jointly the disintegrations of both $\pi$ and $\check{\theta}_\mu$, we define the following measure on $\mathcal{B}$:
$$
\pi \ast \check{\theta}_\mu:=\int_X \int_{\calG} g_\ast^{-1} \check{\theta}^x d\pi^x(g)d\mu(x).
$$

\begin{prop}\label{prop:stationary}
Let $(\mathcal{G},\mu \circ \lambda)$ be a measured groupoid and let $\pi$ be a probability measure equivalent to $\mu \circ \lambda$. Let $(\check{\calB},\check{\theta}_\mu)$ be the Poisson boundary of the Markov operator generated by the reflected measure $\check{\pi}$ with initial distribution $\mu$. Then $\check{\theta}_\mu$ is $\pi$-stationary, that is
$$
\pi \ast \check{\theta}_\mu =\check{\theta}_\mu. 
$$
\end{prop}

\begin{proof}
Let $\check{\mathbf{bnd}}:\Omega_{\calG} \rightarrow \check{\calB}$ be the boundary projection. If $\varphi$ is a bounded Borel function on $\check{\mathcal{B}}$, we have that 
\begin{align*}
\langle \pi \ast \check{\theta}_\mu, \varphi\rangle&=\int_X \int_{\calG} \int_{\check{\calB}} \varphi(g^{-1}b)d\check{\theta}^x(b)d\pi^x(g)d\mu(x)\\
&=\int_X \int_{\calG} \int_{\Omega_{\mathcal{G}}} \varphi(g^{-1}\check{\mathbf{bnd}}(\omega))d\check{\mathbb{P}}_x(\omega)d\pi^x(g)d\mu(x)\\
&=\int_X \int_{\calG} \int_{\Omega_{\mathcal{G}}}\varphi(\check{\mathbf{bnd}}(g^{-1}\omega))d\check{\mathbb{P}}_x(\omega)d\pi^x(g)d\mu(x)\\
&=\int_{\calG} \int_{\Omega_{\mathcal{G}}} \varphi(\check{\mathbf{bnd}}(\omega))d\check{\mathbb{P}}_{g^{-1}}(\omega)d\pi(g)\\
&=\int_{\Omega_{\mathcal{G}}} \varphi(\check{\mathbf{bnd}}(\omega))d\check{\mathbb{P}}_{\check{\pi}}(\omega)\\
&=\int_{\Omega_{\mathcal{G}}} \varphi(\check{\mathbf{bnd}}(\omega))d(T_\ast \check{\mathbb{P}}_{\mu})(\omega)\\
&=\int_{\check{\mathcal{B}}} \varphi(b)d\check{\theta}_\mu(b)=\langle \check{\theta}_\mu, \varphi \rangle,
\end{align*}
where we moved from the first line to the third one using the fact that $\check{\theta}^x=\check{\mathbf{bnd}}_\ast(\check{\mathbb{P}}_x)$ and the equivariance of $\check{\mathbf{bnd}}$, we passed from the third line to the fifth one exploiting that $\check{{\mathbb{P}}}_{g^{-1}}=g^{-1}_\ast \check{\mathbb{P}}_x$, together with the definition of both $\check{\mathbb{P}}_{\check{\pi}}$ and $\check{\pi}$. We concluded by applying the $T$-invariance of the boundary projection $\check{\mathbf{bnd}}$. This proves the statement. 
\end{proof}

The following technical result (compare to \cite[Lemma 2.8]{BF14}) is the last step before the proof of Theorem \ref{theorem_poisson_boundary}. 

\begin{lemma}\label{lemma:poisson}
Retain the setting of Theorem \ref{theorem_poisson_boundary}. Let $\mathcal{E}$ be a positive $\check{\theta}_{\mu} \widehat{\otimes} \theta_{\mu}$-measure subset of $\check{\calB}*\mathcal{B}$ and denote by $\mathcal{A}\coloneqq \check{p}(\mathcal{E})\subset \check{\calB}$, where $\check{p}:\check{\calB}*\mathcal{B}\to \check{\mathcal{B}}$ is the projection on the first component. Then for any $\varepsilon >0$, there exist $x\in X$, $g \in \mathcal{G}^{x}$ and 
$C \subset \check{\mathcal{B}}^x$ of positive 
$\check{\theta}^x$-measure such that $C,g^{-1}C \subset \mathcal{A}$ and for every $b\in C$ it holds that
\begin{equation*}
  \theta^{s(g)}(g^{-1} \mathcal{E}_b)>1-\varepsilon,
\end{equation*}
  where $\mathcal{E}_b\coloneqq \{ b'\in \calB \,|\, (b,b')\in \mathcal{E}\}$.
\end{lemma}
\begin{proof}
We are in the right position to apply Equation \eqref{eq:poisson:fiber:inverse}: for any $x\in X$ and every measurable subset $\mathcal{D}\subset \calB^x$ there exists a bounded $P^x$-harmonic function $h_{\mathcal{D}}$ on $\mathcal{G}^x$ such that 
\begin{equation*}\label{equation_poisson_formula_lemma}
  h_{\mathcal{D}}(g)=\int_{\mathcal{B}^x} \mathbbm{1}_{\mathcal{D}}(b)d\theta^g(b)=\int_{\mathcal{B}^x} \mathbbm{1}_{\mathcal{D}}(b)dg_*\theta^{s(g)}(b)=
\theta^{s(g)}(g^{-1} \mathcal{D}),
\end{equation*}
where $\mathbbm{1}_{\mathcal{D}}$ is the characteristic function on $\mathcal{D}$. By the Martingale Convergence Theorem \cite[Theorem A.3]{BQ16}, for
$\mathbb{P}_x$-almost every $\omega=(g_0,g_1,g_2,\cdots) \in t_{\Omega_{\calG}}^{-1}(x)$ one has that
\begin{equation}\label{equation_convergence_lemma}
 h_{\mathcal{D}}(g_n) \xrightarrow{n\to \infty} \mathbbm{1}_{\mathcal{D}}(\textbf{bnd}(\omega)).
\end{equation}

The set
\begin{align*}
  \Omega_{\mathcal{D}}\coloneqq&\; \{(g_0,g_1,g_2,\cdots)\in \Omega_{\calG} \,|\, h_{\mathcal{D}}(g_n)\xrightarrow{n\to \infty} 1\}   \\
  =&\;\{\omega\in \Omega_{\calG} \,|\,  \mathbbm{1}_{\mathcal{D}}(\textbf{bnd}(\omega)) =1 \}\\
  =&\; \textbf{bnd}^{-1}(\mathcal{D})\,
\end{align*}  
lies inside $t_{\Omega_{\calG}}^{-1}(x)$ and it satisfies 
\begin{equation}\label{equation:1}
\mathbb{P}_x(\Omega_{\mathcal{D}})=\mathbb{P}_x (\textbf{bnd}^{-1}(\mathcal{D}))=\theta^x(\mathcal{D}).
\end{equation}

By the definition of $\mathcal{E}_b$, if $t_{\check{\mathcal{B}}}(b)=x$, we must have that $\mathcal{E}_b \subset t_{\mathcal{B}}^{-1}(x)$. If we denote by $\Omega_{\mathcal{G}} * \check{\mathcal{B}}$ the fiber product with respect to the target maps, we can consider the sets 
\begin{align*}
  \widetilde{\mathcal{E}} \coloneqq & \left\{(\omega,b)\in \Omega_{\calG}*\check{\mathcal{B}}\,|\,
  \omega\in \Omega_{\mathcal{E}_b}\right\}\\
  =&\left\{( \omega,b)\in \Omega_{\mathcal{G}}*\check{\mathcal{B}}\,|\,
 \textbf{bnd}(\omega)\in \mathcal{E}_{b}\right\}\\
  =& \left\{( \omega,b)\in \Omega_{\mathcal{G}}*\check{\mathcal{B}}\,|\,
  (b,\textbf{bnd}(\omega))\in \mathcal{E}\right\}
\end{align*}
and
\begin{align*}
  \widetilde{\mathcal{E}} _N \coloneqq & \left\{(\omega,b)\in \Omega_{\calG}*\check{\mathcal{B}}\,|\,
  \forall \;n>N , \ \; \theta^{s(g_n)}(g_n^{-1}\mathcal{E}_{b})>1-\varepsilon\right\},
\end{align*}
where $g_n$ is the $(n+1)$-th component of $\omega=(g_0,g_1,\ldots) \in \Omega_{\mathcal{G}}$. 

We have 
\begin{align*}
(\mathbb{P}_{\mu} \widehat{\otimes} \check{\theta}_{\mu})(\widetilde{\mathcal{E}})&=\int_{\mathcal{B}}\mathbb{P}_{t_{\check{\calB}}(b)} (\Omega_{\mathcal{E}_b})d\check{\theta}_\mu(b)\\
&=\int_{\mathcal{B}}\mathbb{P}_{t_{\check{\calB}}(b)}(\mathbf{bnd}^{-1}(\mathcal{E}_b))d\check{\theta}_\mu(b)\\
&=\int_{\mathcal{B}}\theta^{t_{\check{\calB}}(b)}(\mathcal{E}_b)d\check{\theta}_{\mu} (b)=(\check{\theta}_{\mu}\widehat{\otimes} \theta_{\mu})(\mathcal{E})>0,
\end{align*}
where we moved from the second line to the third one using Equation \eqref{equation:1}. Moreover, since $\widetilde{\mathcal{E}}_N$ increases to $\widetilde{\mathcal{E}}$ by Equation \eqref{equation_convergence_lemma}, we must have $(\mathbb{P}_{\mu} \widehat{\otimes} \check{\theta}_{\mu})(\widetilde{\mathcal{E}}_N )>0$ for $N$ sufficiently large.

We consider the map
$$\Psi: \Omega_{\mathcal{G}}*\check{\mathcal{B}}\rightarrow \Omega_{\mathcal{G}}*\check{\mathcal{B}},\;\;\;  \Psi(\omega=(g_0,g_1,\ldots),b)\coloneqq (S\omega,g_1^{-1}b)\,,$$
where $S:\Omega_{\mathcal{G}}\rightarrow \Omega_{\mathcal{G}}$ is the one defined by Equation \eqref{equation_map_trajectories}.
Thanks to both Lemma \ref{lemma_preserve} and Proposition \ref{prop:stationary}, the above transformation is probability measure-preserving, namely
$$\Psi_*(\mathbb{P}_{\mu} \widehat{\otimes} \check{\theta}_{\mu})=\mathbb{P}_{\mu} \widehat{\otimes} \check{\theta}_{\mu}\,.$$
Therefore, by applying Poincar\'{e} Recurrence Theorem, we can find $n>N$ such that 
\begin{equation}\label{equation:2}
  (\mathbb{P}_{\mu} \widehat{\otimes} \check{\theta}_{\mu})(\Psi^{-n}(\widetilde{\mathcal{E}} _N)\cap \widetilde{\mathcal{E}}_N)>0.
\end{equation}

We define
$$\mathcal{F}\coloneqq \Psi^{-n}(\widetilde{\mathcal{E}}_N)\cap \widetilde{\mathcal{E}}_N\subset \Omega_{\calG}* \check{\mathcal{B}}$$ 
and
$$\mathcal{F}_{\omega}\coloneqq  \{ b\in \check{\mathcal{B}}\,|\, (\omega,b)\in \mathcal{F}\}\subset \check{\mathcal{B}}^{t_{\Omega_{\calG}}(\omega)}\,.$$
Since the $(\mathbb{P}_{\mu} \widehat{\otimes} \check{\theta}_{\mu})$-measure of $\mathcal{F}$ is positive and 
$$
(\mathbb{P}_{\mu} \widehat{\otimes} \check{\theta}_{\mu})(\mathcal{F})=\int_{\Omega} \check{\theta}^{t_{\Omega_{\calG}}(\omega)}(\mathcal{F}_\omega)d\mathbb{P}_\mu(\omega),
$$
there exists $\omega=(g_0,g_1,g_2,\cdots)\in \Omega_{\calG}$ such that 
$\check{\theta}^{t_{\Omega_{\calG}}(\omega)}(\mathcal{F}_{\omega})>0$. If we choose such an $\omega \in \Omega_{\calG}$, we can set $x=t_{\Omega_{\calG}}(\omega)$, $g=g_n$ the $(n+1)$-th component of $\omega$ and $C=\mathcal{F}_\omega$. With this notation we obtain that $C \subset p_{\check{\calB}}(\widetilde{\mathcal{E}} _N) \subset \check{p}(\mathcal{E})=\mathcal{A}$ and, since $\Psi^n(\mathcal{F})\subset \widetilde{\mathcal{E}}_N$, we have that $g^{-1}C \subset \mathcal{A}$ as well.  

Moreover for any $b \in C$, we have that $(\omega,b) \in \widetilde{\mathcal{E}}_N$, so we get that
$$\theta^{s(g)}(g^{-1}\mathcal{E}_b)=\theta^{s(g_n)}(g_n^{-1} \mathcal{E}_b)>1-\varepsilon,$$
and this concludes the proof. 
\end{proof}

We now have all the necessary ingredients for the proof of the main theorem of the section.

\begin{proof}[Proof of Theorem \ref{theorem_poisson_boundary}]
The amenability of the semidirect groupoids $\calB \rtimes \mathcal{G}$ and $\check{\calB} \rtimes \mathcal{G}$ follows directly by 
Lemma \ref{lem:amenability}. We are left to show the relative metric ergodicity required by Definition \ref{definition_G-boundary}. We will only treat the first-term projection $\check{p}:\check{\mathcal{B}}* \mathcal{B}\rightarrow \check{\mathcal{B}}$, since the proof of the other one is similar.

Consider the following diagram
\begin{center}
  \begin{tikzcd}
  \check{\mathcal{B}}*\mathcal{B}\arrow{r}{f}\arrow{d}[swap]{\check{p}} &  Y\arrow{d}{q} \\
  \check{\mathcal{B}}\arrow{r}{f_0} & Z,
  \end{tikzcd}
  \end{center}
where $Y,Z$ are Borel $\mathcal{G}$-spaces and $q$ is endowed with a fiberwise isometric $\mathcal{G}$-action. 

  For every $b\in \check{\mathcal{B}}$, we define the map
  $$f_b:\mathcal{B}^{t_{\check{\calB}}(b)}\rightarrow Y,\;\;\; f_b(b')\coloneqq f(b,b')$$
and the measure $\beta_b\coloneqq (f_b)_* \theta^{t_{\check{\calB}}(b)}$ obtained by pushing forward $\theta^{t_{\check{\calB}}(b)}$ via $f_b$. Here, since
$\theta^{t_{\check{\calB}}(b)}(\mathcal{B}^{t_{\check{\calB}}(b)})=1$, with a slight abuse of notation we are considering $\theta^{t_{\check{\calB}}(b)}$ as a measure on $\mathcal{B}^{t_{\check{\calB}}(b)}$ instead of a measure on the whole $\mathcal{B}$. 
The commutativity of the above diagram shows that 
$\beta_b(Y^{t_{\check{\calB}}(b)})=1$.

We claim that $\beta_b$ is Dirac for $\check{\theta}_{\mu}$-almost every $b\in \check{\mathcal{B}}$. By contradiction, suppose that
there exists a Borel subset $\mathcal{A}\subset \check{\mathcal{B}}$ of positive $\check{\theta}_{\mu}$-measure such that $\beta_b$ is not Dirac for every $b\in \mathcal{A}$.
As observed by Bader and Furman \cite[Theorem 2.7]{BF14}, this is equivalent to require that there exists $\varepsilon>0$ such that
\begin{equation}\label{equation:proof:thm:1}
  \beta_b(\Ball_{d_b}(f(b,b'),\varepsilon))<1-\varepsilon
\end{equation}
for every $(b,b')\in \mathcal{A}*\mathcal{B}$. Here $d_b$ denotes the metric on $q^{-1}(f_0(b))$. Since for $\check{\theta}_{\mu}\widehat{\otimes} \theta_{\mu}$-almost every $(b,b')$ it holds that
\begin{equation}\label{equation:proof:thm:2}
  \beta_b(\Ball_{d_b}(f(b,b'),\varepsilon))>0\,,
\end{equation}
a Fubini-type argument allows to pick a measurable map 
$$\mathcal{A}\rightarrow \mathcal{B}\,,\;\;\; b\mapsto b'_b$$
such that 
\begin{equation}\label{equation:proof:thm:3}
  \beta_b(\Ball_{d_b}(f(b,b'_b),\varepsilon))>0\,
\end{equation}
for $\check{\theta}_{\mu}$-almost every $b\in \mathcal{A}$. 
Define the set
\begin{align*}
\mathcal{E}\coloneqq & \left\{ (b,b')\in \mathcal{A}*\mathcal{B}\,|\, f_b(b')\in \Ball_{d_b}(f(b,b'_b),\varepsilon) \right\} \\
=&\left\{ (b,b') \in \mathcal{A} \ast \calB \ | \ b' \in f_b^{-1}(\mathrm{Ball}_{d_b}(f(b,b'_b), \varepsilon)) \right\}.
\end{align*}
By definition we have that
\begin{align*}
  \check{\theta}_{\mu} \widehat{\otimes} \theta_{\mu} (\mathcal{E})=& \int_{\mathcal{A}} \theta^{t_{\check{\calB}}(b)} \left(f^{-1}_b(\Ball_{d_b}(f(b,b'_b),\varepsilon))\right) d\check{\theta}_{\mu}(b)\\
  =& \int_{\mathcal{A}} \beta_b \left(\Ball_{d_b}(f(b,b'_b),\varepsilon)\right) d\check{\theta}_{\mu}(b)>0,
\end{align*}
where the latter inequality follows by Equation \eqref{equation:proof:thm:3}.
Thus, we are in the right position to apply Lemma \ref{lemma:poisson} in order to find $x\in X, \, g\in \mathcal{G}^x$ and 
$C\subset \mathcal{A}$ such that $C, g^{-1}C \subset \mathcal{A}$ and
$$
\theta^{s(g)}(g^{-1}\mathcal{E}_b) > 1-\varepsilon,
$$
for every $b \in C$. 

As a consequence, we obtain that
\begin{align*}
 1-\varepsilon<& \;\theta^{s(g)}(g^{-1}\mathcal{E}_b)\\
=& \;\theta^{s(g)}(g^{-1}f^{-1}_b(\mathrm{Ball}_{d_b}(f(b,b'_b),\varepsilon)))\\
=&\;\theta^{s(g)}(f^{-1}_{g^{-1}b}(\mathrm{Ball}_{d_{g^{-1}b}}(f(g^{-1}b,g^{-1}b'_b),\varepsilon)))\\
=&\beta_{g^{-1}b}(\mathrm{Ball}_{d_{g^{-1}b}}(f(g^{-1}b,g^{-1}b_b'),\varepsilon))\\
<&\;1-\varepsilon, 
\end{align*}
and we get the desired contradiction. In the above computation, we passed from the second line to the third one thanks to the equivariance of the fiberwise isometric $\mathcal{G}$-action on $q:Y\rightarrow Z$ and we concluded exploiting the fact that $g^{-1}C \subset \mathcal{A}$. 

Since $\beta_b$ is almost surely Dirac, we can build a Borel map
$$f_1: \check{\mathcal{B}} \rightarrow Y$$
so that $\delta_{f_1(b)}=\beta_b$ for almost every $b\in \check{\mathcal{B}}$.
The fact that $f_1$ is a map of $\mathcal{G}$-spaces and the equality $q \circ f_1=f_0$ follows by
$\beta_b(q^{-1}(f_0(b)))=1$.
Moreover, $f_1$ is $\mathcal{G}$-equivariant, in fact
$$\delta_{f_1(gb)}=\beta_{gb}=g_*\beta_b=g_* \delta_{f_1(b)}=\delta_{gf_1(b)}\,,$$
and this concludes the proof.
\end{proof}

\section{Algebraic representability of an ergodic measured groupoid and equivariant maps}\label{section_algebraic_representability}

In this section we study representations of ergodic groupoids into algebraic groups. We first list some properties of measurable equivariant maps and we give a reduction criterion for measurable representations of an ergodic groupoid into a locally compact group. Afterwards, following Bader and Furman \cite{BF:Unpub}, we extend the notion of algebraic representability of an ergodic action to any ergodic groupoid. All this machinery will be crucial in the study of equivariant maps from boundaries. 

\subsection{A representation reduction lemma} In this section we focus our attention on equivariant maps with respect to a representation of an ergodic groupoid $(\mathcal{G},\nu)$ into a locally compact group $H$. We first show that the existence of an equivariant map into a homogeneous quotient of $H$ guarantees that the representation admits a representative in its similarity class whose image is contained in a proper subgroup of $H$.  We recall that the \emph{essential image} of a representation $\rho:\mathcal{G} \rightarrow H$ is the support in $H$ of the pushforward measure $\rho_\ast(\nu)$. 

\begin{lemma}\label{lem:equivariant:map}
Let $(\mathcal{G},\nu)$ be a measured groupoid with units $X$. Let $H$ be a locally compact second countable group and consider $\rho:\mathcal{G} \rightarrow H$ a measurable representation. Then there exists a measurable $\rho$-equivariant function $X \rightarrow H/L$ if and only if there exists a representation which is similar to $\rho$ and whose essential image is contained in a closed subgroup $L<H$. 
\end{lemma}
\begin{proof}
Suppose that there exists a measurable representation 
$$
\overline{\rho}: \mathcal{G} \rightarrow H,
$$
which is similar to $\rho$ and whose essential image lies in a closed subgroup $L$. Then, there exists a measurable map $h:X \rightarrow H$ such that 
\begin{equation}\label{eq:homotopy}
h(t(g))\overline{\rho}(g)=\rho(g)h(s(g)),
\end{equation}
for almost every $g \in \mathcal{G}$. If we compose $h$ with the quotient projection $H \rightarrow H/L$ we obtain a map
$$
\overline{h}:X \rightarrow H/L,
$$
which satisfies 
$$
\overline{h}(t(g))=\rho(g)\overline{h}(s(g)),
$$
for almost every $g \in \mathcal{G}$ by Equation \eqref{eq:homotopy}. Equivalently, $\overline{h}$ is a measurable map which is $\rho$-equivariant, as claimed. 

Suppose now that there exists a Borel map $\overline{h}:X \rightarrow H/L$ which is $\rho$-equivariant, where $L$ is a closed subgroup. By \cite[Corollary A.8]{zimmer:libro} there exists a Borel section $H/L \rightarrow H$ with respect to the quotient projection. By composing such a section with $\overline{h}$ we obtain a measurable map $h:X \rightarrow H$. 

Since $\overline{h}$ is $\rho$-equivariant, we have that
$$
\overline{h}(t(g))=\rho(g)\overline{h}(s(g)),
$$
for almost every $g \in \mathcal{G}$. This means that $h(t(g))$ and $\rho(g)h(s(g))$ lies in the same $L$-coset, or equivalently
$$
\overline{\rho}(g):=h(t(g))^{-1}\rho(g)h(s(g)) \in L,
$$
for almost every $g \in \mathcal{G}$. This proves the statement and concludes the proof. 
\end{proof}

Thanks to the previous lemma, we can give a reducibility criterion for representations of an ergodic groupoid.

\begin{prop}\label{prop:reduction}
Let $(\mathcal{G},\nu)$ be an ergodic measured groupoid with units $X$. Let $H$ be a locally compact second countable group acting continuously and smoothly on a topological space $Y$. Consider a measurable representation $\rho:\mathcal{G} \rightarrow H$. If there exists a $\rho$-equivariant Borel map $\psi:X \rightarrow Y$, then $\rho$ is similar to a representation whose essential image is contained in the stabilizer $\mathrm{Stab}_H(y_0)$ for some $y_0 \in Y$. 
\end{prop}

\begin{proof}
We can compose $\psi$ with the quotient projection $p:Y \rightarrow Y/H$ to obtain a Borel map
$$
\overline{\psi}:X \rightarrow Y/H.
$$
Since $\psi$ is $\rho$-equivariant, we have that 
$$
\overline{\psi}(t(g))=\overline{\psi}(s(g)),
$$
for almost every $g \in \mathcal{G}$. The smoothness of the $H$-action on $Y$ guarantees that the quotient $Y/H$ is countably separated. The ergodicity of $\mathcal{G}$ implies that $\overline{\psi}$ is essentially constant. As a consequence the essential image of $X$ is contained in a unique $H$-orbit. Namely, there exists some $y_0 \in Y$ such that we can restrict the map as follows 
$$
\psi:X \rightarrow Hy_0, 
$$
up to modifying the original map on a set of measure zero. By \cite[Theorem 2.1.14]{zimmer:libro} we have an identification 
$$
Hy_0 \rightarrow H/\mathrm{Stab}_H(y_0).
$$
By composing $\psi$ with the above identification and the measurable section \cite[Corollary A.8]{zimmer:libro}
$$
H/\mathrm{Stab}_H(y_0) \rightarrow H
$$
to the quotient map, we obtain a map measurable function
$$
h:X \rightarrow H. 
$$
We claim that $h$ realizes the desired similarity. In fact, if we define
$$
\overline{\rho}(g):=h(t(g))^{-1}\rho(g)h(s(g)),
$$
we must have that 
$$
\overline{\rho}(g) \in \mathrm{Stab}_H(y_0)
$$
for almost every $g \in \mathcal{G}$, because $h(t(g))$ and $\rho(g)h(s(g))$ are mapped to the same $\mathrm{Stab}_H(y_0)$-orbit by the $\rho$-equivariance of $\psi$. Thus the essential image of $\overline{\rho}$ is contained in $\mathrm{Stab}_H(y_0)$ and we are done. 
\end{proof}

\subsection{Algebraic representations of ergodic measured groupoids} In this section we extend the theory of algebraic representability by Bader and Furman \cite{BF:Unpub} to ergodic groupoids. 

\begin{defn}\label{def:algebraic:representability}
Let $\kappa$ be a field with a non-trivial absolute value so that the induced topology is complete and separable. Let $(\mathcal{G},\nu)$ be an ergodic measured groupoid with unit space $X$. Consider a measurable representation $\rho:\mathcal{G} \rightarrow H$ into the $\kappa$-points $H=\mathbf{H}(\kappa)$ of an algebraic $\kappa$-group $\mathbf{H}$. An \emph{algebraic representation of $(\mathcal{G},\nu)$ relative to $\rho$} is the datum of an algebraic $\kappa$-variety $\mathbf{V}$ with an algebraic $\textbf{H}$-action and of a Borel map 
$$
\varphi:X \rightarrow \mathbf{V}(\kappa),
$$
which is $\rho$-equivariant, namely
$$
\varphi(t(g))=\rho(g)\varphi(s(g)),
$$
for almost every $g \in \mathcal{G}$. 
\end{defn}

Given an algebraic representation of $(\mathcal{G},\nu)$ relative to a measurable representation $\rho:\mathcal{G} \rightarrow H$, we will refer to it by $(\mathbf{V},\varphi_{\mathbf{V}})$, where $\mathbf{V}$ is the algebraic variety which appears as co-domain in the equivariant map $\varphi_{\mathbf{V}}$. A \emph{morphism} between two such algebraic representations $(\mathbf{V},\varphi_{\mathbf{V}})$ and $(\mathbf{U},\varphi_{\mathbf{U}})$ of $(\mathcal{G},\nu)$ is an algebraic $\kappa$-map $\Psi:\mathbf{V} \rightarrow \mathbf{U}$ which is $\mathbf{H}$-equivariant and it satisfies $\varphi_{\mathbf{U}}=\Psi \circ \varphi_{\mathbf{V}}$. We say that an algebraic representation $(\mathbf{V},\varphi_{\mathbf{V}})$ is of \emph{coset type} if there exists some algebraic $\kappa$-subgroup $\mathbf{L}<\mathbf{H}$ such that $\mathbf{V}=\mathbf{H}/\mathbf{L}$. 

As in the case of actions \cite[Proposition 5.2]{BF:Unpub}, we start showing that it is always possible to substitute an algebraic representation with a representation of coset type. 

\begin{prop}\label{prop:coset}
Let $(\mathcal{G},\nu)$ be an ergodic groupoid. Let $\rho:\mathcal{G} \rightarrow H$ be a measurable representation into the $\kappa$-points $H=\mathbf{H}(\kappa)$ of an algebraic $\kappa$-group. Consider $(\mathbf{V},\varphi_{\mathbf{V}})$ an algebraic representation of $(\mathcal{G},\nu)$ relative to a measurable representation $\rho:\mathcal{G} \rightarrow H$. Then, there exists an algebraic $\kappa$-subgroup $\mathbf{L}<\mathbf{H}$ and a morphism of algebraic representations between $\mathbf{H}/\mathbf{L}$ and $\mathbf{V}$. 
\end{prop}

\begin{proof}
We denote by $V=\mathbf{V}(\kappa)$ the $\kappa$-points of $\mathbf{V}$. By \cite[Proposition 2.1]{BF:Unpub} we know that the $H$-orbits in $V$ are locally closed. By the Glimm-Effros Theorem \cite[Theorem 2.1.14]{zimmer:libro} the $H$-action on $V$ is smooth. By composing $\varphi_{\mathbf{V}}$ with the quotient projection $V \rightarrow V/H$, we obtain a map
$$
\overline{\varphi}_{\mathbf{V}}:X \rightarrow V/H
$$
which is $\mathcal{G}$-invariant, namely
$$
\overline{\varphi}_{\mathbf{V}}(t(g))=\overline{\varphi}_{\mathbf{V}}(s(g)),
$$
for almost every $g \in \mathcal{G}$. The smoothness of the $H$-action guarantees that $V/H$ is countably separated and the ergodicity of $(\mathcal{G},\nu)$ implies that $\overline{\varphi}_{\mathbf{V}}$ is essentially constant. Equivalently, $\varphi_{\mathbf{V}}$ takes essentially values in a unique $H$-orbit in $V$.
Up to modifying $\varphi_{\mathbf{V}}$ on a null set, we can suppose that there exists some $v_0 \in V$ such that 
$$
\varphi_{\mathbf{V}}:X \rightarrow H\cdot v_0.
$$ 
Again by smoothness \cite[Theorem 2.1.14]{zimmer:libro}, we can identify
$$
j:H\cdot v_0 \rightarrow H/L, 
$$
where $L=\mathrm{Stab}_H(v_0)$. We define $\Phi:=j \circ \varphi_{\mathbf{V}}$. If we denote by $i:H\cdot v_0 \rightarrow V$ the orbit inclusion, we obtain a commutative diagram
$$
\begin{tikzcd}
X\arrow{rr}{\varphi_{\mathbf{V}}}\arrow{rrdd}[swap]{\Phi} && H\cdot v_0\arrow{rr}{i} \arrow{dd}{j} && V\\
\\
&&H/L \arrow{rruu}[swap]{\Psi}&&,
\end{tikzcd}
$$
where $\Psi:=i \circ j^{-1} $. By \cite[Proposition 2.1]{BF:Unpub} the group $L$ corresponds to the $\kappa$-points of $\mathbf{L}$, the stabilizer of $v$ in $\mathbf{H}$, and the map $\Psi$ can be extended to an algebraic $\kappa$-map $\Psi:\mathbf{H}/\mathbf{L} \rightarrow \mathbf{V}$. To conclude is sufficient to compose $\Phi$ with the natural embedding $H/L \rightarrow \mathbf{H}/\mathbf{L}(\kappa)$. 
\end{proof}

We continue by showing the existence of an initial object of coset type in the category of algebraic representations of an ergodic groupoid. 

\begin{thm}\label{teor:initial:object}
Let $(\mathcal{G},\nu)$ be an ergodic groupoid and let $\rho:\mathcal{G} \rightarrow H$ be a measurable representation in the $\kappa$-point of an algebraic $\kappa$-group $\mathbf{H}$. Then the category of algebraic representations of $(\mathcal{G},\nu)$ relative to $\rho$  admits an initial object of coset type. 
\end{thm}

\begin{proof}
We define the collection 
$$
\{ \mathbf{L} < \mathbf{H} \ | \ \textup{$\mathbf{L}$ is defined over $\kappa$ and there exists a coset representation to $\mathbf{H}/\mathbf{L}$}\}
$$
which is clearly not empty, since it contains $\mathbf{H}$. Thus, by Noetherianity, it admits a minimal element $\mathbf{L}_0$. We denote the corresponding equivariant map by $\varphi_0:X \rightarrow (\mathbf{H}/\mathbf{L}_0)(\kappa)$. We want to show that the pair $(\mathbf{H}/\mathbf{L}_0,\varphi_0)$ is the desired initial object. 

Let $(\mathbf{V},\varphi_{\mathbf{V}})$ be an algebraic representation of $(\mathcal{G},\nu)$ relative to $\rho$. We want to show that there exists a morphism from $\mathbf{H}/\mathbf{L}_0$ to $\mathbf{V}$ (the uniqueness is guaranteed by the fact that two different morphisms must agree nowhere). We consider the product representation $(\mathbf{V} \times \mathbf{H}/\mathbf{L}_0, \varphi:=\varphi_{\mathbf{V}} \times \varphi_0)$. Thanks to Proposition \ref{prop:coset}, we must have the following commutative diagram
$$
\begin{tikzcd}
Y \arrow[bend left=40]{ddrrrr}{\varphi_0} \arrow{rr} \arrow{dd}{\varphi_{\mathbf{V}}} \arrow{ddrr}{\varphi} && (\mathbf{H}/\mathbf{L})(\kappa) \arrow{dd}{i} &&\\
\\
\mathbf{V}(\kappa) && \mathbf{V}(\kappa) \times (\mathbf{H}/\mathbf{L}_0)(\kappa) \arrow{ll}{p_1} \arrow{rr}[swap]{p_2} && (\mathbf{H}/\mathbf{L}_0)(\kappa).
\end{tikzcd}
$$
By the minimality of $\mathbf{L}_0$, we must have that the composition $p_2 \circ i:\mathbf{H}/\mathbf{L} \rightarrow \mathbf{H}/\mathbf{L}_0$ is actually a $\kappa$-isomorphism. Thus we can define the following $\kappa$-morphism 
$$
p_1 \circ i \circ (p_2 \circ i)^{-1}:\mathbf{H}/\mathbf{L}_0 \rightarrow \mathbf{V},
$$
and the statement is proved. 
\end{proof}

\begin{defn}\label{def:algebraic:gate}
Let $(\calG,\nu)$ be an ergodic groupoid and consider a measurable representation $\rho:\calG \rightarrow H$ into the $\kappa$-points $H=\mathbf{H}(\kappa)$ of some algebraic $\kappa$-group $\mathbf{H}$. Let $(\mathbf{H}/\mathbf{L},\varphi)$ be the initial object of coset type in the category of algebraic representations given by Theorem \ref{teor:initial:object}. We say that the pair $(\mathbf{L},\varphi)$ is the \emph{algebraic gate} associated to the representation $\rho$ and we call $\mathbf{L}$ its \emph{algebraic hull}. We say that $\rho$ is \emph{Zariski dense} if  $\mathbf{L}=\mathbf{H}$ (and then $\varphi$ is the constant map). 
\end{defn}

The notion of algebraic hull we gave for $\rho$ it is a clear generalization of the one introduced by Zimmer \cite[Chapter 9.2]{zimmer:libro} in the context of measurable cocycles. If the starting groupoid is also amenable, we can show that its algebraic hull is an amenable group, generalizing in this way \cite[Theorem 9.2.3]{zimmer:libro}.

\begin{thm}\label{teor:amenable:hull}
Let $(\mathcal{G},\nu)$ be an ergodic groupoid. Consider a measurable representation $\rho:\mathcal{G} \rightarrow H$ into the real points of a connected real algebraic group $\mathbf{H}$. If the groupoid is amenable, then the algebraic hull of $\rho$ is an amenable subgroup of $H$. 
\end{thm}

\begin{proof}
We denote by $\mathbf{N}$ the radical of $\mathbf{H}$, so that the quotient $\mathbf{Q}=\mathbf{H}/\mathbf{N}$ is semisimple. Let $\mathbf{P}$ be a minimal parabolic $\mathbb{R}$-subgroup of $\mathbf{Q}$. We denote by $P=\mathbf{P}(\mathbb{R})$ the real points of $\mathbf{P}$ and likewise for all the other groups.

Since we supposed the groupoid $\mathcal{G}$ amenable, by \cite[Theorem 4.2.7]{delaroche:renault} there exists an equivariant probability measure valued Borel map $\psi:X \rightarrow \mathcal{M}^1(Q/P)$. Recall that the action of $Q$ on the space of probability measures $\mathcal{M}^1(Q/P)$ is smooth by \cite[Corollary 3.2.17]{zimmer:libro} and the stabilizers are solvable (and thus amenable) and algebraic \cite[Corollary 3.2.22]{zimmer:libro}. The fact that $N$ acts trivially on $Q/P$ implies that $H$ acts smoothly on $\mathcal{M}^1(Q/P)$ as well. Moreover, since $N$ is amenable, stabilizers in $H$ will be amenable by \cite[Proposition 4.1.6]{zimmer:libro}. We are in the right position to apply Proposition \ref{prop:reduction}: we obtain that $\rho$ is similar to a representation essentially contained in the $H$-stabilizer of some probability measure on $Q/P$. By the amenability of that stabilizer and by the minimality of the algebraic hull the statement follows.
\end{proof}

We are finally ready to prove

\begin{rec_thm}[\ref{thm_equivariant_map}]
  Let $\rho:\mathcal{G} \rightarrow H$ be any measurable representation of an ergodic groupoid into the $\kappa$-points of an algebraic $\kappa$-group $\mathbf{H}$. Let $((\mathcal{B}_-,\theta_-),(\mathcal{B}_+,\theta_+))$ be a boundary pair for $\calG$. Then there exist $\kappa$-subgroups $\mathbf{L}_\pm<\mathbf{H}$ and measurable $\rho$-equivariant maps $\mathcal{B}_{\pm} \rightarrow H/L_\pm$, where $L_\pm=\mathbf{L}_\pm(\kappa)$. 
\end{rec_thm}
\begin{proof}
We start noticing that the measurable representation $\rho:\mathcal{G} \rightarrow H$ can be actually extended to the semidirect groupoids $\mathcal{B}_\pm \rtimes \calG$ by forgetting the $\mathcal{B}_\pm$-variable. 

By Theorem \ref{teor:initial:object}, it is sufficient to show that the groupoids $\mathcal{B}_\pm \rtimes \calG$ are ergodic. By Proposition \ref{proposition_properties} (ii) applied to the projections 
$$\mathcal{B}_-*\mathcal{B}_+\rightarrow \mathcal{B}_\pm,$$
we know that $\mathcal{B}_\pm$ is isometrically ergodic. By the ergodicity of $(\calG,\nu)$ and Proposition \ref{proposition_properties} (iii), the semidirect products $\mathcal{B}_\pm \rtimes \calG$ are ergodic, as claimed. 
\end{proof}

\begin{oss}\label{oss:fiber:isometrically}
It is worth noticing that also the semidirect product $(\mathcal{B}_- \ast \mathcal{B}_+) \rtimes \calG$ is ergodic. Let $X$ be the unit space of $\calG$, as usual. 
Since $\mathcal{B}_\pm$ is isometrically ergodic and the projections $\mathcal{B}_- \ast \mathcal{B}_+ \rightarrow \mathcal{B}_\pm$ are relatively isometrically ergodic, we get that $\mathcal{B}_- \ast \mathcal{B}_+$ is isometrically ergodic as well. Then the groupoid $(\mathcal{B}_- \ast \mathcal{B}_+) \rtimes \calG$ is ergodic by Proposition \ref{proposition_properties} (iii). 

By extending the measurable representation $\rho:\calG \rightarrow H$ to the groupoid $(\mathcal{B}_- \ast \mathcal{B}_+) \rtimes \calG$, we obtain an equivariant measurable map $\mathcal{B}_- \ast \mathcal{B}_+ \rightarrow H/L_0$, where $L_0=\mathbf{L}_0(\kappa)$ for some $\kappa$-subgroup $\mathbf{L}_0<\mathbf{H}$. 
\end{oss}

\subsection{The Zariski dense case}\label{section_furstenberg_maps}
In this final section we focus our attention on Zariski dense representations of an ergodic groupoid $\mathcal{G}$ into a real algebraic group $H$. We will give an existence results for \emph{Furstenberg maps}, namely equivariant maps between the boundaries of $\mathcal{G}$ and $H$. In this way, we obtain a generalization of both \cite[Theorem 3.1]{BF14} and \cite[Theorem 1]{sarti:savini:3}.

We will need the following preliminary result about the slices of an equivariant map (cfr. \cite[Proposition 4.4]{sarti:savini:3}).

\begin{lemma}\label{lemma:Zarski:density:slices}
Let $(\mathcal{G},\nu)$ be an ergodic groupoid and let $\rho:\calG \rightarrow H$ be a Zariski dense representation into the real points $H=\mathbf{H}(\mathbb{R})$ of a real algebraic group. Let $((\mathcal{B}_-,\theta_-),(\mathcal{B}_+,\theta_+))$ be a boundary pair for $\mathcal{G}$. Suppose that there exist real algebraic subgroups $\mathbf{L}_\pm<\mathbf{H}$ and $\rho$-equivariant measurable maps $\phi_\pm:\mathcal{B}_\pm \rightarrow H/L_\pm$, Then the essential image of the restriction $\phi^x_\pm$ to the fiber $\mathcal{B}^x_\pm=t_{\mathcal{B}_\pm}^{-1}(x)$ is Zariski dense for almost every $x\in X=\mathcal{G}^{(0)}$.
\end{lemma}
\begin{proof}
We are going to do the proof only in the case of $(\calB_+,\theta_+)$, since the other one is completely analogous.

Let $\{\theta^x_+\}_{x \in X}$ the $t_{{\calB}_+}$-disintegration of the measure $\theta_+$ on $\calB_+$. For every $x\in X$, we denote by $\mathbf{V}^x_+$ the \emph{essential Zariski closure} of $\phi^x_+$, namely the smallest Zariski closed subset of $\mathbf{H}/\mathbf{L}_+$ such that $\theta^x_+((\phi^x_+)^{-1}(\mathbf{V}^x_+(\mathbb{R})))=1$.  The existence of those sets is guaranteed by Noetherianity, as already observed by Pozzetti \cite{Pozzetti}. By Chevalley's Theorem \cite[Proposition 3.1.4]{zimmer:libro} we can embed $\mathbf{H}/\mathbf{L}_+$ in a suitable projective space through a rational homomorphism $\mathbf{H} \rightarrow \mathrm{GL}(N+1,\mathbb{C})$ defined over the reals. In this way we obtain a map 
$$\Phi_+: X\rightarrow \text{Var}_{\mathbb{R}}(\mathbb{P}^N)\,,\;\;\; \Phi_+(x)\coloneqq \mathbf{V}^x_+,$$
where $\mathrm{Var}_{\mathbb{R}}$ denotes the space of all closed subvarieties in $\mathbb{P}^N$ defined over $\mathbb{R}$, endowed with the Hausdorff topology. This is well-defined since Zariski closed sets are closed in the analytic topology. 

We claim that $\Phi_+$ is measurable. By Hahn disintegration Theorem \cite[Theorem 2.1]{Hahn}, the slice $\phi^x_+$ depends measurably on $x$. Here we are tacitly viewing $\mathrm{L}^0(\mathcal{B}_+,H/L_+)$ as the space of sections of the measurable bundle of metric spaces $\{L^0(\mathcal{B}^x_+,H/L_+)\}_{x \in X}$. As a consequence, the coefficients of any homogeneous polynomial defining $\mathbf{V}^x_+$ depends measurably on $x$.  Finally, the fact that the zero locus associated to an homogeneous ideal varies measurably with respect to the coefficients guarantees that $\Phi_+$ is measurable. Hence, the claim follows.

Such map $\Phi_+$ is also $\rho$-equivariant, namely
$$\mathbf{V}^{t(g)}_+=\rho(g)\mathbf{V}^{s(g)}_+,$$
where the $H$-action on $\textup{Var}_{\mathbb{R}}(\mathbb{P}^N)$ is induced by the natural left-action of $\textup{GL}(N+1,\mathbb{C})$ on $\textup{Var}_{\mathbb{R}}(\mathbb{P}^N)$. As noticed by Zimmer \cite[Proposition 3.3.2]{zimmer:libro}, the set $\mathrm{Var}_{\mathbb{R}}(\mathbb{P}^N)$ decomposes as a countable union of (subsets of) varieties on which $\mathrm{GL}(N+1,\mathbb{C})$ acts algebraically and hence smoothly \cite[Theorem 3.1.3, Theorem 2.1.14]{zimmer:libro}. By Proposition \ref{prop:reduction} the representation $\rho$ must be similar to a representation preserving some fixed Zariski closed subvariety $\mathbf{V}_0$ of $\mathbf{H}/\mathbf{L}_+$. By Zariski density of $\rho$ we argue that 
$$
\mathbf{V}_0=\mathbf{H}/\mathbf{L}_+,
$$
so almost every slice of $\phi$ is essentially Zariski dense, as claimed. 
\end{proof}

We are now ready to prove the main result of this section.

\begin{rec_thm}[\ref{theorem_furstenberg_map}]
  Let $\rho: \mathcal{G}\rightarrow H$ be a measurable Zariski dense representation of an ergodic groupoid $(\calG,\nu)$ into the real points $\text{H}=\textbf{\textup{H}}(\mathbb{R})$ of a real algebraic group. Let $((\mathcal{B}_-,\theta_-),(\mathcal{B}_+,\theta_+))$ be a boundary pair for $\mathcal{G}$. Then there exist $\mathcal{G}$-equivariant measurable maps $\mathcal{B}_\pm \rightarrow H/P_\pm$, where $P_\pm<H$ are minimal parabolic subgroups.

  \noindent In particular if $(\mathcal{B},\theta)$ is a $\mathcal{G}$-boundary, then there exists an equivariant measurable map $\mathcal{B} \rightarrow H/P$.
\end{rec_thm}

\begin{proof}


By Theorem \ref{thm_equivariant_map} and Remark \ref{oss:fiber:isometrically} there exist $\phi_\pm: \mathcal{B}_\pm \rightarrow H/L_\pm $ and $\phi_0:\mathcal{B}_-*\mathcal{B}_+\rightarrow H/L_0$ initial objects in the category of algebraic representations relative to $\rho$ for $\mathcal{B}_\pm \rtimes \mathcal{G}$ and for $(\mathcal{B}_-*\mathcal{B}_+)\rtimes \mathcal{G}$, respectively. Here $L_\pm=\textbf{L}_\pm(\mathbb{R})<H$, $L_0=\textbf{L}_0(\mathbb{R})<H$ are real points of $\textbf{L}_\pm,\textbf{L}_0<\textbf{H}$ suitable algebraic subgroups. Additionally, the amenability of the groupoids $\calB_\pm \rtimes \calG$ implies that $L_\pm$ are amenable by Theorem \ref{teor:amenable:hull}. 

By looking at the diagram 
\begin{equation}\label{equation_diagram_phi0}
 \begin{tikzcd}
   \mathcal{B}_-*\mathcal{B}_+ \arrow{rr}{\phi_0} \arrow{d}[swap]{p_+} && H/L_0 \arrow[dotted]{d}\\
   \mathcal{B}_+ \arrow{rr}[swap]{\phi} && H/L_+\,,
 \end{tikzcd}   
\end{equation}
the universal property of $\phi_0$ implies the existence of a $\mathcal{G}$-map $H/L_0\rightarrow H/L_+$. Thus, up to conjugation, we obtain $L_0<L_+$. If $R_+\coloneqq \text{Rad}_u(L_+)$ denotes the unipotent radical of $L_+$, we get the chain of inclusion 
$$L_0<L_0R_+<L_+$$
and the associated chain of projections
$$H/L_0\xrightarrow{q'_+} H/L_0R_+\xrightarrow{q_+} H/L_+\,.$$
Consider now the following commutative square
\begin{equation}\label{equation_diagram5}
 \begin{tikzcd}
   \mathcal{B}_-*\mathcal{B}_+\arrow{rr}{\widehat{q'_+\circ \phi_0}}\arrow{d}[swap]{p_+} &&
   H/L_0R_+\times X\arrow{d}{q_+ \times \id_X} \\
  \mathcal{B}_+\arrow{rr}{\widehat{\phi}_+}&& H/L_+\times X\,.
 \end{tikzcd}   
\end{equation}
where $\widehat{q'_+ \circ \phi_0}(b,b')=(q'_+\circ \phi_0(b,b'),t_{\mathcal{B}_+}(b))$ and $\widehat{\phi}_+(b)=(\phi_+(b),t_{\mathcal{B}_+}(b))$. The right vertical arrow admits a fiberwise isometric action, as observed by Bader and Furman \cite[Theorem 3.4]{BF14}. Indeed the fibers $(q_+ \times \mathrm{id}_X)^{-1}(L_+,x)$ can be identified with $L_+/L_0R_+$, and the latter admits a natural $L_+/R_+$-invariant metric. Using the $H$-action we can extend that metric to all the other fibers to get a fiberwise isometric action of $\mathcal{G}$. Therefore, we are in the right setting to apply the relative isometric ergodicity of $p_+$. Precisely, we obtain a $\mathcal{G}$-map $\phi'_+:\mathcal{B}_+\rightarrow H/L_0R_+$ such that $q_+\circ \phi'_+=\phi_+$. On the other hand, being $\phi_+$ an initial object, there exists a right-inverse $\widetilde{q}_+:H/L_+\rightarrow H/L_0R_+$ for $q_+$, which implies that $q_+$ is an isomorphism.
Hence, up to conjugation, we can assume that $L_0R_+=L_+$. The same will hold also in the other case, namely $L_0R_-=L_-$. 

Consider now the map 
$$\phi_-\times \phi_+:\mathcal{B}_-*\mathcal{B}_+\rightarrow H/L_-\times H/L_+\,,\;\;\; (\phi_-\times \phi_+)(b_1,b_2)\coloneqq (\phi_-(b_1),\phi_+(b_2))\,,$$ that fits in the commutative triangle
\begin{center}
 \begin{tikzcd}
 \mathcal{B}_-*\mathcal{B}_+\arrow{rr}{\phi_0}\arrow{ddr}[swap]{\phi_-\times \phi_+} & &H/L_0\arrow[dotted]{ldd}[]{\widetilde{\phi}} \\& &\\
  &H/L_-\times H/L_+&
 \end{tikzcd}
 \end{center}
where the existence of $\widetilde{\phi}$ follows since $\phi_0$ is initial. 
By Lemma \ref{lemma:Zarski:density:slices} almost every slice $\phi^x_\pm$ is Zariski dense in $H/L_\pm$, so the product $\phi^x_-\times \phi^x_+$ is Zariski dense in $H/L_-\times H/L_+$ as well. This implies the Zariski density of the essential image of $\widetilde{\phi}$. Equivalently, $R_-L_0R_+$ is dense in $H$. Hence, by \cite[Lemma 3.5]{BF14} the subgroups $L_+$ and $L_-$ are parabolic and, being amenable, they are also minimal. This concludes the proof.
\end{proof}

\section*{Competing interests}
The authors declare none.

\bibliographystyle{alpha}

\bibliography{biblionote}

\end{document}